\documentclass[11pt,reqno]{amsart}
\usepackage{geometry}
\usepackage{xcolor}
\usepackage{amsthm}
\usepackage{amssymb}
\usepackage{hyperref}
\usepackage{arydshln}
\usepackage{bbm}
\usepackage{mathrsfs}

\makeatletter
\numberwithin{equation}{section}
\numberwithin{figure}{section}

\newcommand{\C}{\mathbb{C}}
\newcommand{\R}{\mathbb{R}}
\newcommand{\N}{\mathbb{N}}

\newcommand{\F}{\mathcal{F}}

\renewcommand{\P}{\mathbb{P}}
\newcommand{\E}{\mathbb{E}}
\newcommand{\e}{\varepsilon}

\newcommand{\1}{\mathbbm{1}}

\newtheorem{Theorem}{Theorem}[section]
\newtheorem{Proposition}[Theorem]{Proposition}\newtheorem{Corollary}[Theorem]{Corollary}\newtheorem{Lemma}[Theorem]{Lemma}\newtheorem{Remark}[Theorem]{Remark}\newtheorem{Example}[Theorem]{Example}

\numberwithin{equation}{section}

\makeatother

\begin{document}
\title[The Volterra square-root process]{Volterra square-root process: Stationarity and regularity of the law}
\author{Martin Friesen}
\author{Peng Jin}

\address[Martin Friesen]{School of Mathematical Sciences\\
Dublin City University\\ Glasnevin, Dublin 9, Ireland}

\email{martin.friesen@dcu.ie}

\address[Peng Jin]{Division of Science and Technology \\ BNU-HKBU United International College \\ Zhuhai, China; Department of Mathematics \\ Shantou University \\ Shantou, Guangdong 515063, China}

\email{pengjin@uic.edu.cn}

\date{\today}



\begin{abstract}\   The Volterra square-root process on $\R_+^m$ is an affine Volterra process with continuous sample paths. Under a suitable integrability condition on the resolvent of the second kind associated with the Volterra convolution kernel, we establish the existence of limiting distributions. In contrast to the classical square-root diffusion process, here the limiting distributions may depend on the initial state of the process.  Our result shows that the non-uniqueness of limiting distributions is closely related to the integrability of the Volterra convolution kernel. Using an extension of the exponential-affine transformation formula we also give the construction of stationary processes associated with the limiting distributions. Finally, we prove that the time-marginals as well as the limiting distributions, when restricted to the interior of the state space $\R_{+}^m$, are absolutely continuous with respect to the Lebesgue measure and their densities belong to some weighted Besov space of type $B_{1,\infty}^{\lambda}$.
\end{abstract}

\maketitle

\allowdisplaybreaks

\section{Introduction}

\subsection{General introduction}

The analysis performed in \cite{MR3805308} on intraday stock market data suggests that the volatility, seen as a stochastic process, has sample paths of very low regularity and hence is not adequately captured by existing models such as the Heston model. Moreover, classical well-established Markovian models are often not able to capture the observed term structure of at-the-money volatility skew.
To accommodate for these features, the authors propose to work with rough analogues of stochastic volatility models. The most prominent example is the \textit{rough Heston model} as studied in \cite{MR3778355,MR3861827,MR3905737}. We also refer interested readers to \cite{MR4031338,HAN2021101580,MR4283941, MR4229250, MR4332695, MR4205882} for some recent developments of this model. Extensions to multi-factor settings and general Volterra kernels have been studied in \cite{MR4019885}, where general affine Volterra processes with continuous sample paths have been constructed. Other extensions to multi-asset settings with rough correlations are recently studied in \cite{MR4370808, MR4031333}. While the rough sample path behavior observed in \cite{MR3805308} still remains controversial (see, e.g., \cite{FTW19, F21, CD22}), the newly emerged rough volatility models have proven themselves to fit the empirical data remarkably well. 

In this work we study the Volterra square-root process which provides the most general example of a continuous affine Volterra process on $\R_+^m$. We investigate their limiting distributions, construction of the stationary process, and absolute continuity of the law. Below we first recall the definition of the Volterra square-root process, and then discuss our results and the related literature. The $m$-dimensional \textit{Volterra square-root process} $X = (X_t)_{t \geq 0}$ is obtained from the stochastic Volterra equation
\begin{align}\label{eq: VCIR}
 X_t = x_0 + \int_0^t K(t-s)\left( b + \beta X_s \right)ds + \int_0^t K(t-s)\sigma(X_s)dB_s,
\end{align}
where $\sigma(x) = \mathrm{diag}(\sigma_1 \sqrt{x_1}, \dots, \sigma_m \sqrt{x_m})$, $x_0 \in \R_+^m$, $K \in L_{loc}^2(\R_+; \R^{m \times m})$, and $(B_t)_{t \geq 0}$ is a standard Brownian motion on $\R^m$.
Here and below we call $(b,\beta,\sigma, K)$ \textit{admissible}, if they satisfy
\begin{enumerate}
    \item[(i)] $b \in \R_+^m$;
    \item[(ii)] $\beta = (\beta_{ij})_{i,j = 1,\dots, m} \in \R^{m\times m}$ is such that
    $\beta_{ij} \geq 0$ for all $i\neq j$;
    \item[(iii)] $\sigma = (\sigma_1,\dots, \sigma_m)^{\top} \in \R_+^m$;
    \item[(iv)] The kernel $K$ is diagonal with $K = \mathrm{diag}(K_1,\dots, K_m)$, where the scalar kernels $K_{i}\in L_{\mathrm{loc}}^{2}(\mathbb{R}_{+},\mathbb{R})$, $i = 1,\dots, m$;
    \item[(v)]There exist constants $\gamma \in (0,2]$ and $C_1 > 0$ such that
    \[
     \int_{0}^{h}|K_i(r)|^{2}dr\leq C_1h^{\gamma},\qquad h\in[0,1], \ \ i = 1,\dots, m,
    \]
    and for each $T > 0$ there exists $C_2(T) > 0$ such that
    \[
     \int_{0}^{T}|K_i(r+h)-K_i(r)|^{2}dr \leq C_2(T) h^{\gamma}, \qquad h \in [0,1], \ \ i = 1, \dots, m;
    \]
    \item[(vi)] For each $i = 1,\dots, m$ and each $h \in [0,1]$ the shifted kernel $t \longmapsto K_i(t+h)$
    is nonnegative, not identically zero, nonincreasing and continuous on $(0,\infty)$, and its resolvent of the first kind $L_i$ is nonnegative and nonincreasing in the sense that $s \longmapsto L_i([s,s+t])$ is nonincreasing for all $t \geq 0$.
\end{enumerate}
Note that for nonnegative, nonincreasing, and not identically zero functions $K_1,\dots, K_m$ the resolvent of the first kind always exists, see \cite[Chapter 5, Theorem 5.5]{MR1050319}. That is, there exist measures $L_1,\dots, L_m$ of locally bounded variation on $\R_+$ such that $K_i \ast L_i = L_i \ast K_i = 1$, $i=1,\ldots,m$, where $\ast$ denotes the usual convolution of functions or measures on $\R_+$.
\begin{Remark}
 If $K_1, \dots, K_m$ are completely monotone and not identically zero, then condition (vi) holds, see \cite[Chapter 5, Theorem 5.4]{MR1050319}.
\end{Remark}
For a discussion of condition (v) we refer to \cite{MR4019885}.
Given admissible parameters $(b,\beta,\sigma, K)$ it follows from
\cite[Theorem 6.1]{MR4019885} that for each $x_0 \in \R_+^m$ there exists a unique (in law) $\R_+^m$-valued weak solution $X = (X_t)_{t \geq 0}$ of \eqref{eq: VCIR} which is defined on some stochastic basis $(\Omega, \F, (\F_t)_{t\geq 0}, \mathbb{P})$ supporting an $m$-dimensional Brownian motion $(B_t)_{t \geq 0}$. Moreover, for each $\eta \in (0,\gamma/2)$, $X$ has a modification with $\eta$-H\"older continuous sample paths and satisfies $\sup_{t \in [0,T]}\E[|X_t|^p] < \infty$ for each $p \geq 2$ and $T >0$.
\begin{Example}\label{example: CIR}
 In dimension $m = 1$ with $K(t) = t^{H -1/2}/\Gamma(H + 1/2)$, we recover the \textit{rough Cox-Ingersoll-Ross process} which reads as
\begin{align*}
  X_t &= x_0 + \int_0^t \frac{(t-s)^{H-\frac{1}{2}}}{\Gamma(H + \frac{1}{2})} (b + \beta X_s) ds + \sigma \int_0^t \frac{(t-s)^{H-\frac{1}{2}}}{\Gamma(H + \frac{1}{2})}\sqrt{X_s}dB_s.
\end{align*}
 Here $(B_t)_{t \geq 0}$ is a one-dimensional Brownian motion, $x_0, b, \sigma \geq 0$, $\beta \in \R$, and $H \in (0,1/2)$. For this kernel $K$ one can choose $\gamma=2H$ in the above definition of admissible parameters, see \cite[Example 2.3]{MR4019885}. The process $(X_t)_{t\ge0}$ is neither a finite-dimensional Markov process nor a semimartingale which makes its mathematical study an interesting task.
\end{Example}

Similarly to the classical square-root process, there is a semi-explicit form for the Fourier-Laplace transform of the Volterra square-root process, i.e., it is an affine process on $\R_+^m$. To state this formula in a compact form, let us define
\[
  R_i(u) = \langle u, \beta^i \rangle + \frac{\sigma_i^2}{2}u_i^2, \qquad i=1,\ldots, m, \ \ u \in \C^m,
\]
where $\beta^i = (\beta_{1i}, \dots, \beta_{mi})^{\top}$ denotes the $i$-th column of the matrix $\beta$. Let $R = (R_1,\dots, R_m)^{\top}$ and set
$\C_-^m = \{ u \in \C^m \ : \ \mathrm{Re}(u) \leq 0 \}$.
It follows from \cite[Theorem 6.1]{MR4019885} that for each $u \in \C_-^m$ and $f \in L_{loc}^1(\R_+; \C_-^m)$
the system of Riccati-Volterra equations
\begin{align}\label{eq: riccati}
 \psi(t) = K(t)u + \int_0^t K(t-s)f(s)ds +\int_{0}^{t}K(t-s)R(\psi(s))ds
\end{align}
has a unique global solution $\psi = \psi(\cdot,u,f) \in L_{loc}^2(\R_+; \C_-^m)$. The Volterra square-root process satisfies the exponential-affine transformation formula for the Fourier-Laplace transform
\begin{align}\label{eq: fourier transform}
 &\ \mathbb{E}\left[\mathrm{e}^{\langle u,X_{t}\rangle + \int_0^t \langle X_{t-s},f(s) \rangle ds}\right]
 \\ &= \exp\left\{ \langle u, x_{0}\rangle + \int_0^t \langle f(s),x_0\rangle ds + \int_0^t \langle x_0, R(\psi(s)) \rangle ds + \int_0^t \langle b, \psi(s) \rangle ds \right\}. \notag
\end{align}
Here and below we let, by slight abuse of notation, $\langle z,w \rangle:=\sum_{i=1}^m z_iw_i$ for $z,w\in\mathbb{C}^m$. Note that $\langle \cdot, \cdot \rangle$ does not correspond to the usual inner product on $\C^m$.

\subsection{Long-time behavior}

Mean-reversion is a commonly accepted stylized fact in stochastic volatility modeling. Mathematically, this feature may be captured by the notion of ergodicity for the volatility process, i.e., by limit distributions and stationarity.
Even more so stationarity also plays an important role for statistical inference.
Namely, if one makes discrete time observations, then stationarity guarantees that these samples can be drawn from the same invariant distribution. Hence one may  estimate the parameters through the invariant distribution, see e.g. \cite{MR3526353} where this was done for the subcritical Heston model. For other related results and implications for applications we refer to \cite{MR2599675, MR3216637, MR3035268}. 

It is well-known that the classical CIR process (that is, Example \ref{example: CIR} with $H = 1/2$) is mean-reverting with long-term mean $-b \beta^{-1}$ and speed of mean-reversion $\beta$. Mathematically this can be justified by studying ergodicity of the process. More precisely, if $\beta < 0$, the expected value satisfies $\E[X_t] \to b |\beta|^{-1}$ as $t \to \infty$ and the process has a unique limiting distribution $\pi$ which is stationary, see also \cite{2018arXiv181205402J, 2019arXiv190105815F} where ergodicity of more general affine processes on the canonical state space were studied.
Convergence results in the stronger total variation distance and hence the law-of-large numbers have been studied in \cite{MR4102262, FJKR22} for general classes of affine processes.

In this work we provide a sufficient condition for the existence of limiting distributions of the Volterra square-root process. Moreover, we characterize all limiting distributions and show that each limiting distribution gives rise to a stationary process, hence showing that the process is indeed mean-reverting. In contrast to the classical CIR process, the limiting distributions of the Volterra square-root process may  depend on the initial state $x_0$ (even if $m=1$, and $\beta$ is negative or $b = 0$). We also characterize the case where the limiting distribution  $\pi_{x_0}$ actually depends on the initial state $x_0$, and prove that all limiting distributions satisfy
$\pi_{x_0} = \pi_{Px_0} = \pi_0 \ast \pi_{Px_0}^{b=0}$.
Here $P$ is a certain projection operator, and $\pi_{Px_0}^{b=0}$ denotes the limiting distribution of the Volterra square-root process with initial state $Px_0$ and $b = 0$. Our proof is essentially based on an analysis of the Riccati-Volterra equation \eqref{eq: riccati} reformulated for the Laplace transform. The latter one is more suitable for the cone structure of $\R_+^m$, see Theorem \ref{Theorem: limiting distribution} and Theorem \ref{Theorem: limiting distribution characterisation}.

For the existence of limiting distributions it suffices to show that the limit $t \to \infty$ in \eqref{eq: fourier transform} exists and is continuous at $u = 0$ and then to apply L\'evy's continuity theorem as done in \cite{2018arXiv181205402J}. This requires to show that $\psi$ obtained from \eqref{eq: riccati} has additionally global integrability in time (e.g., $\psi \in L^1(\R_+; \C^{m}) \cap L^2(\R_+; \C^{m})$), which is studied in Section 3. However, to prove the existence of an associated stationary process we cannot rely (in contrast to the literature) on the use of the Feller semigroup and classical Markovian techniques. In this work we propose an alternative approach based on an extension of the exponential-affine transformation formula. Namely, we prove in Section 3 that formulas \eqref{eq: riccati} and \eqref{eq: fourier transform} can be extended from $(u,f)$ to a general class of vector-valued measures $\mu$. This allows us to effectively track the finite dimensional distributions of the process. Hence, we can show that as $h\to \infty$, the process $(X^h_t)_{t\ge0}$ defined by $X^h_t = X_{t+h}$ converges in law to a continuous process $X^{\mathrm{stat}}$ which is the desired stationary process. Let us mention that this approach seems to be  new even for classical affine processes.

At this point it is worthwhile to mention that some of the Markovian methods can be recovered once the process is lifted to an infinite dimensional Markov process described by a generalized Feller semigroup in the sense of \cite{MR4181950}. Such an approach was recently used in \cite{JPS22} to tacke a similar problem to the one studied in this work. There the authors study a class of continuous affine Volterra processes on the canonical state-space $D = \R^n \times \R_+^m$ where the first $n$-components correspond to the $\log$-asset prices while the last $m$-components are the stochastic volatility factors. In this framework, the authors prove under certain subcriticallity conditions on the drift and for fractional kernels of type $K(t) = t^{\alpha-1}/\Gamma(\alpha)$ with $\alpha \in (0,1)$ the existence of a stationary process for the Markovian lift. Uniqueness and the characterization of stationary processes is left open. Finally, applying their result to our setting, that is $n = 0$, we see that the volatility factors are essentially one-dimensional rough CIR processes as given in Example \ref{example: CIR} with the additional restriction $b = 0$. Their key method is based on the observation that uniform boundedness of the first moment translates to uniform boundedness of the operator norm of the generalized Feller semigroup which itself is shown to be sufficient for the existence of an invariant measure. In contrast, our methods are based on a detailed study of the Riccati-Volterra equation and allow us to study the multi-dimensional case for a general class of Volterra kernels $K$, allowing for $b \neq 0$, and also prove the uniqueness of stationary processes by providing a characterisation of their finite dimensional distributions.

Our results are in line with the existing literature on limiting distributions for stochastic Volterra equations. In \cite{BDK19} the authors studied limiting distributions for L\'evy driven Volterra SPDEs. Using the Markovian lift onto the Filipovic space they have shown that for kernels $K$ being elements of the Filipovic space multiple limiting distributions may occur. At the same time, in \cite{MR4241464} an abstract SPDE framework was provided to deal with SPDEs having multiple limiting distributions. This framework covers the Markovian lift onto the Filipovic space as well as stochastic delay equations. While both works cover rather general classes of Volterra stochastic equations, they require that the kernel is sufficiently regular which excludes e.g. $K(t) = \frac{t^{H-1/2}}{\Gamma(H+1/2)}$ with $H \in (0,1)$ and hence cannot be applied to the Volterra square-root process.

\subsection{Regularity of the law}

In the second part of this work we turn to the study of regularity of the law of $X_t$ for fixed $t > 0$. This includes absolute continuity of the law as well as regularity of the density. In the case of classical affine processes, such results can lead to the strong Feller property of the process (see \cite{MR4130578}). We also want to point out that better regularity of the density are known for classical affine processes, see, e.g., (see \cite{MR3084047}), where the authors studied density approximations and their applications in mathematical finance, and obtained  $C^k$-regularity of the density up to the boundary. It is still unclear whether similar results can be obtained for affine Volterra processes.

In this work we prove that, when $\sigma_1, \dots, \sigma_m > 0$ and $K$ satisfies a suitable lower bound, the distribution of $X_t$ in \eqref{eq: VCIR} is absolutely continuous with respect to the Lebesgue measure on $\R_{++}^m := \R_+^m \backslash \partial \R_+^m$, the interior of its state space. Moreover, our proof shows that the density $p_t(x)$ satisfies $\min\{1, x_1^{1/2},\dots, x_m^{1/2}\}p_t(x) \in B_{1,\infty}^{\lambda}(\R^m)$, where $B_{1,\infty}^{\lambda}(\R^m)$ denotes the Besov space of order $(1,\infty)$ and some $\lambda \in (0,1)$ denoting the regularity of the function. In particular, by Sobolev embeddings, the density has some low $L^p(\R^m)$-regularity. Moreover, assuming that the H\"older increments of $K$ satisfy a global estimate in the spirit of condition (v) of the admissible parameters (see condition (K)), we also show that the limiting distributions $\pi_{x_0}$ are absolutely continuous with respect to the Lebesgue measure on $\R_{++}^m$ and
have the same regularity as the law of $X_t$ with $t > 0$. For a precise version of this results we refer to Section 6.

Our proof is based on a method that was first introduced in \cite{DF13} and subsequently applied in \cite{MR4130578,MR4127334} to continuous-state branching processes with immigration.
While the aforementioned works aimed to prove existence of a transition density for a Markovian SDE,
in this work we extend this method to Volterra stochastic equations and, additionally, demonstrate that this method
also can be applied for the limiting distributions.
Note that, our method here is different from existing methods to study densities of classical affine processes which are often based on estimates of the characteristic function, see \cite{MR3084047} and  \cite{FJKR22}.

\subsection{Application to  the Volterra CIR process with a Gamma kernel}

In this section we briefly state our results when applied to the \textit{Volterra Cox-Ingersoll-Ross  process with a Gamma kernel} obtained from
\begin{align}\label{eq: rough CIR}
  X_t &= x_0 + \int_0^t K(t-s) (b + \beta X_s) ds + \sigma \int_0^t K(t-s)\sqrt{X_s}dB_s,
\end{align}
where $K(t) = \frac{t^{H-\frac{1}{2}}}{\Gamma(H + \frac{1}{2})}e^{-\lambda t}$,
$H \in (0,1/2)$, $\lambda, \sigma, b, x_0 \geq 0$, $\beta \in \R$, and $(B_t)_{t \geq 0}$ is a one-dimensional Brownian motion. The following is our main result on limiting distributions and stationarity of the process.
\begin{Theorem}\label{Theorem: VCIR 1dim}
 Let $X$ be obtained from \eqref{eq: rough CIR} and $\beta < \lambda^{H+1/2}$.
 Then the process $(X_{t+h})_{t\ge0}$ converges in law to a continuous stationary process $(X_t^{\mathrm{stat}})_{t \geq 0}$ when $h \to \infty$. Moreover, the finite dimensional distributions of $X^{\mathrm{stat}}$ have the characteristic function
 \begin{align*}
    \E\left[ e^{ \sum_{j=1}^n \langle u_j,X^{\mathrm{stat}}_{t_j} \rangle } \right]
     &= \exp\left( \frac{\lambda^{H+1/2}x_0 + b}{\lambda^{H+1/2} - \beta} \left( \sum_{k=1}^n u_k + \frac{\sigma^2}{2}\int_0^{\infty}\psi(s)^2ds \right) \right),
 \end{align*}
 where $0 \leq t_1 < \dots < t_n$, $u_1,\dots, u_n \in \C_-$, and $\psi$ is the unique solution of
 \begin{align*}
  \psi(t) &= \sum_{j=1}^n \1_{\{ t >  t_n - t_j\}}K(t -(t_n - t_j))u_j
  \\ &\qquad \qquad + \int_0^{t}K(t-s)\left( \beta \psi(s) + \frac{\sigma^2}{2}\psi(s)^2\right)ds.
 \end{align*}
 Moreover, the first moment and the autocovariance function of the stationary process satisfy
 \[
  \E[X_t^{\mathrm{stat}}] = \frac{\lambda^{H + 1/2}x_0 + b}{\lambda^{H+1/2} - \beta}
 \]
 and if additionally $\beta < 0$ and $\sigma > 0$, then for $0 \leq s \leq t$,
 \begin{equation}\label{asym autocovariance}
  \mathrm{cov}(X_t^{\mathrm{stat}}, X_s^{\mathrm{stat}}) \,\asymp\, (t-s)^{-(H+3/2)}e^{-\lambda (t-s)}, \qquad t-s \to \infty.
 \end{equation}
\end{Theorem}
 Here and after, if $f,g$ are positive functions, the notation $f \asymp g$ means  that  there is a positive constant $c$ such that $ c^{-1}g\leq f\leq cg$. As a consequence of our results, we see that the stationary process $X^{\mathrm{stat}}$ is independent of the initial state $x_0$ if and only if $\lambda = 0$. Moreover, since for $H=1/2$ the autocovariance function satisfies
$\mathrm{cov}(X_t^{\mathrm{stat}}, X_s^{\mathrm{stat}}) \,\asymp\, e^{-(\lambda + |\beta|)(t-s)}$ which can be seen by direct computation for the classical CIR process, we find that for $\lambda = 0$ the autocovariance function has a phase transition from power-law to exponential decay when $H \nearrow 1/2$.
Our result implies, in particular, that $X_t$ converges weakly to some limiting distribution $\pi_{x_0}$ when $t \to \infty$, and that its characteristic function is given by the expression in Theorem \ref{Theorem: VCIR 1dim} with $n = 1$ and $\psi$ being determined from \eqref{eq: riccati}.
Note that in contrast to the classical CIR case where $H = 1/2$ and $\lambda=0$, the limiting distribution may satisfy $\pi_{x_0} \neq \delta_0$ even when $b = 0$.
At this point we would also like to mention the recent work \cite{FGS22}, which provides a heuristic argument on the existence of limit distributions and the form of its Fourier-transform for the Volterra CIR process.

The following is our main result on the regularity of the law when applied to the \textit{Volterra CIR process with a Gamma kernel}.
\begin{Theorem}
 Let $X$ be the Volterra CIR process with a Gamma kernel obtained from \eqref{eq: rough CIR}. Suppose that $\sigma > 0$. Then the following assertions hold:
 \begin{enumerate}
     \item[(a)] There exists some non-negative function $  p_t \in L^1(\R_+)$ such that
     \[
      \mathcal{L}(X_t)(dx) = \P[X_t = 0]\delta_0(dx) + p_t(x)dx, \qquad \forall t > 0.
     \]
     Let $p_t^*(x) = \1_{\R_+}(x)\min\{1, \sqrt{x}\}p_t(x)$, then there exists  $\lambda \in (0,1)$ and another constant $C > 0$ independent of $t$ such that for $h \in [-1,1]$ and $t > 0$,
     \[
      \int_{0}^{\infty}|p_t^*(x+h) - p_t^*(x)|dx  \leq C|h|^{\lambda} (1 \wedge t)^{-H}.
     \]

     \item[(b)] Let $\pi_{x_0}$ be the limiting distribution of $X$. Then
     \[
      \pi_{x_0}(dx) = \pi_{x_0}(\{0\})\delta_0(dx) + \rho_{x_0}(x)dx
     \]
     for some $0 \leq \rho_{x_0} \in L^1(\R_+)$.
     Letting $\rho_{x_0}^*(x) = \1_{\R_+}(x)\min\{1,\sqrt{x}\}\rho_{x_0}(x)$, then
     \[
      \int_0^{\infty}|\rho_{x_0}^*(x+h) - \rho_{x_0}^*(x)|dx \leq C|h|^{\lambda}, \qquad h \in [-1,1],
     \]
     for the same $\lambda$ as in part (a) and some constant $C > 0$.
 \end{enumerate}
\end{Theorem}
As a consequence, apart from the origin the process $X$ has a density in the interior of the state space. The above is a special case of our results from Section 6 which are applicable to arbitrary dimensions and a large class of kernels $K$. One implication of this regularity is related with convergence in total variation to the limiting distribution. Indeed, noting that $p_t^* \longrightarrow \min\{1,\sqrt{x}\}\pi_{x_0}(x)$ weakly as $t \to \infty$, and that $(p_t^*)_{t \geq 1} \subset L^1(\R_+)$ is relatively compact due to $\sup_{t \geq 1}\| p_t^*\|_{B_{1,\infty}^{\lambda}} < \infty$, we conclude that
$p_t^* \longrightarrow \min\{1,\sqrt{x}\}\pi_{x_0}(x)$ in $L^1(\R_+)$ and hence in total variation.

These findings motivate us to study boundary non-attainment for the Volterra CIR process with a Gamma kernel, i.e., conditions for $\P[ X_t = 0, \ \ \forall t > 0] = 1$. In the classical case when $H = 1/2$ and $\lambda=0$, such problem is related with the Feller condition to be imposed on $\sigma$ and $b$. Its rough analogue is left for future research.

\subsection{Structure of the work}
This work is organized as follows. In Section 2 we collect some preliminaries for the study of Volterra integral equations. In Section 3 we study the Riccati-Volterra equation \eqref{eq: riccati} and establish some regularity results for its solution $\psi(t,u,f)$. In Section 4 we first prove global bounds for the moments of the Volterra square-root process. Moreover, we show that the H\"older increments are uniformly bounded in $L^2$. Based on the moment bounds, we prove in Section 5 the weak convergence of the law of $X_t$ to $\pi_{x_0}$ when $t \to \infty$, construct the stationary process, and finally provide a characterization when $\pi_{x_0}$ actually depends on $x_0$. In Section 6 we prove for a large class of kernels $K$ that the distribution of $X_t$ as well as $\pi_{x_0}$ is absolutely continuous with respect to the Lebesgue measure on $\R_{++}^m$.

\section{Preliminaries on Volterra integral equations}

\subsection{Convolution on $\R_+$}
For $p \in [1,\infty]$ we let $L^p([0,T]; \C^m)$ be the Banach space of equivalence classes of functions $f: [0,T] \longrightarrow \C^m$ with finite norm
$\| f\|_{L^p([0,T])}$. Similarly we define $L^p([0,T]; \C^{n \times k})$ with $n,k \in \N$ as the Banach space of matrix-valued functions $f: [0,T] \longrightarrow \C^{n \times k}$,
where $\C^{n \times k}$ is equipped with the operatornorm $\|A\|_{2} = \sup_{|v|=1}|Av|$ with respect to the Euclidean distances on $\C^n$ and $\C^k$.
We denote by $L_{loc}^p(\R_+; \C^m) = L_{loc}^p$ and $L_{loc}^p(\R_+; \C^{n \times k})$ the spaces of locally $p$-integrable functions.

The $m \times m$ identity matrix is denoted by $I_m$. Let $\|\cdot\|_{HS}$
be defined by $\|A\|_{HS}=\sqrt{\mathrm{tr}(A^{*}A)}$
be the Hilbert-Schmidt norm on $\C^{m\times m}$. Note that $\sqrt{m}^{-1}\| A\|_{HS} \leq \|A\|_2 \leq \|A\|_{HS}$ and both norms are sub-multiplicative in the sense that $\| AB \|_2 \leq \|A\|_2\|B\|_2$ and $\|AB\|_{HS} \leq \|A\|_{HS}\|B\|_{HS}$. Let $T > 0$.
The convolution of two functions $f: [0,T] \longrightarrow \C^{n \times k}, \ g: [0,T] \longrightarrow \C^{k \times d}$ is defined by
\[
 f\ast g(t) = \int_0^t f(t-s)g(s)ds, \qquad t \in [0,T]
\]
with the matrix multiplication under the integral. We frequently use Young's inequality which states that $\| f\ast g \|_{L^r([0,T])} \leq \| f\|_{L^p([0,T])}\| g\|_{L^q([0,T])}$ whenever for $p,q,r \in [1,\infty]$ with $\frac{1}{p} + \frac{1}{q} = \frac{1}{r} + 1$ the right-hand side is finite.
\begin{Remark}\label{rem: continuity}
 If $f \in L^p([0,T])$ and $g \in L^q([0,T])$ with $p,q \in [1,\infty]$ such that $\frac{1}{p}+\frac{1}{q}=1$, then $f\ast g$ is continuous. To see this, set $\tilde{f}(x)=f(x)\mathbf{1}_{[0,T]}(x)$, $\tilde{g}(x)=g(x)\mathbf{1}_{[0,T]}(x)$ and then apply \cite[2.20 Lemma]{MR1817225}.
\end{Remark}

We also use the convolution of a function and a measure. Namely, let $\mathcal{M}_{lf}$ be the space of all $\C^m$-valued set functions $\mu$ on $\R_+$ for which the restriction $\mu|_{[0,T]}$ with $T>0$ is a $\C^m$-valued finite measure.
For $\mu \in \mathcal{M}_{lf}$ and a compact set $E\subset \R_+$ recall that
\[
|\mu|(E):=\sup\bigg\{\sum_{j=1}^{N}|\mu(E_j)|: \, \{E_j\}_{j=1}^{N} \mbox{ is a measurable partition of } E\bigg\}.
\]
Given $f: (0,T] \longrightarrow \C^{m\times m}$
and $\mu \in \mathcal{M}_{lf}$, we define their convolution by
\[
 f \ast \mu(t) = \int_{[0,t)} f(t-s)\mu(ds), \qquad t \in (0,T]
\]
provided it is well-defined. It is easy to see that for each $p \in [1,\infty]$, $\| f \ast \mu\|_{L^p([0,T])} \leq \|f \|_{L^p([0,T])}|\mu|([0,T])$.
Moreover, if $f$ is continuous on $[0,T]$, then $f \ast \mu$ is also continuous on $[0,T]$.

\subsection{Volterra integral equations}

Let $K: \R_+ \longrightarrow \C^{m\times m}$ be locally integrable, i.e., $K \in L_{loc}^1(\R_+; \C^{m \times m})$. Consider for given $h \in L_{loc}^1(\R_+; \C^m)$ and $B \in \C^{m\times m}$ the Volterra convolution equation
\[
 x(t) = h(t) + \int_0^t K(t-s)B x(s)ds.
\]
Note that this equation is equivalent to
$x + K_{B} \ast x = h$,
where $K_{B}(t) = -K(t)B$.
According to \cite[Chapter 2, Theorem 3.5]{MR1050319}, it has a unique solution $x \in L_{loc}^1(\R_+; \C^m)$ given by $x = h - R_{B} \ast h$,
where $R_{B} \in L_{loc}^1(\R_+; \C^m)$ is the \textit{resolvent of the second kind} of the kernel $K_{B}$ defined by the relation
\begin{align}\label{eq: def Rbeta}
 R_{B} \ast K_{B} = K_{B} \ast R_{B} = K_{B} - R_{B}.
\end{align}
Note that $K_{B} \in L_{loc}^1(\R_+; \C^m)$ guarantees that such a function $R_{B}$  always exists, see
\cite[Chapter 2, Theorem 3.1]{MR1050319}. If the function $h$ is of the form $h = K \ast \mu$ with $\mu \in \mathcal{M}_{lf}$, then the unique solution takes the form $x = K \ast \mu - R_{B} \ast K \ast \mu = E_{B} \ast \mu$,
where we have set
\begin{align}\label{eq: def Ebeta}
 E_{B} = K-R_{B}\ast K.
\end{align}
Note that $E_{B}(-B) = R_{B}$.
\begin{Remark}\label{integrability of RB and EB}
 By \eqref{eq: def Rbeta}, \eqref{eq: def Ebeta}, and Young's inequality one has $R_{B}, E_{B} \in L_{\mathrm{loc}}^p(\R_+; \C^{m\times m})$, whenever $K_{B} \in L_{\mathrm{loc}}^p(\R_+; \C^{m \times m})$ with $p\ge1$. Moreover, if $K$ is continuous on $(0,\infty)$, then  $R_{B}$ and $E_{B}$ are also continuous on $(0,\infty)$.
\end{Remark}

\begin{Remark}
 If $K = I_m$, then $E_{B}(t) = e^{B t}$ and $R_{B}(t) = (-B)e^{B t}$.
 In this case integrability of $E_{B}$ on $\mathbb{R}_+$ is equivalent to $B$ having only eigenvalues with strictly negative real parts.
\end{Remark}

\section{Analysis of the Volterrra-Riccati equation}

\subsection{Regularity in time}

In the next lemma and its proof we will use the notation $R_{\beta^{\top}}$ and $E_{\beta^{\top}}$ which are respectively defined  by \eqref{eq: def Rbeta} and \eqref{eq: def Ebeta} with $B=\beta^\top$.
Using different methods, we will later on see that $R_{\beta} = (R_{\beta^{\top}})^{\top}$ and $E_{\beta} = (E_{\beta^{\top}})^{\top}$.

\begin{Lemma}\label{lemma: local integrability psi}
 Let $(b,\beta, \sigma, K)$ be admissible parameters and let $\psi$ be the unique solution of \eqref{eq: riccati} in $L_{loc}^2(\R_+; \C_-^{m})$, where $u \in \C_-^m$ and $f \in L_{loc}^1(\R_+; \C_-^{m})$. If $p \in [1,\infty]$ is such that $K \in L_{loc}^p(\R_+; \R^{m\times m})$, then for each $T > 0$,
 \begin{align*}
  &\ \| \psi(\cdot, u, f) \|_{L^p([0,T])} \leq 2\left( |u| + \| f\|_{L^1([0,T])} \right)\| E_{\beta^{\top}}\|_{L^p([0,T])}
  \\ &\qquad \qquad + \left( \sum_{i=1}^m \frac{\sigma_i^2}{2} \right)\left(|u| + \| f\|_{L^1([0,T])}\right)^2 \| E_{\beta^{\top}}\|_{L^p([0,T])} \| E_{\beta^{\top}}\|_{L^2([0,T])}^2.
 \end{align*}
 Moreover, if $\mathrm{Im}(u) = \mathrm{Im}(f) = 0$, then
 \[
  \| \psi(\cdot,u,f)\|_{L^p([0,T])} \leq \left(|u| + \| f\|_{L^1([0,T])} \right) \| E_{\beta^{\top}} \|_{L^p([0,T])}.
 \]
\end{Lemma}
\begin{proof}
 It follows from the proof of \cite[Lemma 6.3]{MR4019885} that the real and imaginary parts of $\psi$ satisfy for each $i = 1,\dots, m$ the inequalities
 $\ell_i(t) \leq \mathrm{Re}(\psi_i(t,u,f)) \leq 0$ and $|\mathrm{Im}(\psi_i(t,u,f))| \leq h_i(t)$,
where the functions $\ell_i$ and $h_i$ are the unique global solutions of
\begin{align*}
    h_i(t) &= K_i(t)|\mathrm{Im}(u_i)| + \int_0^t K_i(t-s)\left( |\mathrm{Im}(f_i(s))| +  \langle h(s), \beta^i \rangle \right) ds
    \\ \ell_i(t) &= K_i(t)\mathrm{Re}(u_i) + \int_0^tK_i(t-s)\left( \mathrm{Re}(f_i(s)) + \langle \ell(s),\beta^i \rangle - \frac{\sigma_i^2}{2}h_i(s)^2\right) ds
\end{align*}
and $\beta^i = (\beta_{1i}, \dots, \beta_{mi})^{\top}$ denotes the $i$-th column of $\beta$. Hence we obtain
$\| \psi(\cdot, u,f)\|_{L^{p}([0,T])}
 \leq \| \ell \|_{L^{p}([0,T])} + \| h \|_{L^{p}([0,T])}$, where $\ell=(l_1,\ldots,l_m)^{\top}$ and $h=(h_1,\ldots,h_m)^{\top}$. To estimate the right-hand side of the previous inequality, we define for $i=1,\dots, m$
 \begin{align*}
    h_{0,i}(t) &= K_i(t)|\mathrm{Im}(u_i)| + \int_0^t K_i(t-s)|\mathrm{Im}(f_i(s))|ds,
    \\ \ell_{0,i}(t) &= K_i(t)\mathrm{Re}(u_i) + \int_0^t K_i(t-s)\mathrm{Re}(f_i(s))ds - \frac{\sigma_i^2}{2}\int_0^t K_i(t-s)h_i(s)^2ds.
 \end{align*}
 Then $h(t) = h_0(t) + \int_0^t K(t-s)\beta^{\top} h(s)ds$
and $\ell(t) = \ell_0(t) + \int_0^t K(t-s)\beta^{\top} \ell(s)ds$. This gives
$h(t) = h_0(t) - (R_{\beta^{\top}} \ast h_0)(t)$ and hence
\begin{align*}
 h(t) = E_{\beta^{\top}}(t)(|\mathrm{Im}(u_1)|,\ldots,|\mathrm{Im}(u_m)|)^{\top} + E_{\beta^{\top}} \ast (|\mathrm{Im}(f_1)|,\ldots,|\mathrm{Im}(f_m)|)^{\top}(t).
\end{align*}
Likewise we obtain $\ell(t) = \ell_0(t) - (R_{\beta^{\top}} \ast \ell_0)(t)$ and hence
\begin{align*}
    \ell(t) = E_{\beta^{\top}}(t)\mathrm{Re}(u) + E_{\beta^{\top}} \ast \left( \mathrm{Re}(f) - \varkappa \right)(t),
\end{align*}
where $\varkappa_i(t) = \frac{\sigma_i^2}{2}h_i(t)^2$.
Young's inequality yields
\[
 \| h \|_{L^p([0,T])} \leq \left( |u| + \|f\|_{L^1([0,T])}\right) \| E_{\beta^{\top}} \|_{L^p([0,T])}
\]
and
\[
 \| \ell \|_{L^p([0,T])} \leq \left( |u| + \| f\|_{L^1([0,T])} + \| \varkappa \|_{L^1([0,T])}\right)\|E_{\beta^{\top}}\|_{L^p([0,T])}.
\]
Estimating
\begin{align*}
 \| \varkappa \|_{L^1([0,T])} &\leq \bigg(\sum_{i=1}^{m}\frac{\sigma_i^2}{2}\bigg) \| h\|_{L^2([0,T])}^2
 \\ &\leq \left( \sum_{i=1}^{m}\frac{\sigma_i^2}{2} \right) \left( |u| + \|f\|_{L^1([0,T])}\right)^2 \| E_{\beta^{\top}} \|_{L^2([0,T])}^2
\end{align*}
 readily yields the first assertion.
 For the second assertion note that $\mathrm{Im}(u) = \mathrm{Im}(f) = 0$ implies $\mathrm{Im}(\psi) = h = 0$ and hence
 $\ell(t) = E_{\beta^{\top}}(t)u + E_{\beta^{\top}}\ast f(t)$.
 The second estimate is now a consequence of Young's inequality.
\end{proof}

The next result provides pointwise regularity of $\psi$ with respect to the time variable.
\begin{Theorem}\label{Theorem regularity 1}
 Let $(b,\beta, \sigma, K)$ be admissible parameters and let $\psi$ be the unique solution of \eqref{eq: riccati} in $L_{loc}^2(\R_+; \C_-^{m})$, where $u \in \C_-^m$ and $f \in L_{loc}^1(\R_+; \C_-^{m})$. Then the following assertions hold:
 \begin{enumerate}
     \item[(a)] If $f \in L_{loc}^{3/2}(\R_+; \C_-^m)$ and $K \in L_{loc}^3(\R_+; \R^{m\times m})$, then $\psi(\cdot,u,f)$ is continuous on $(0,\infty)$. If $u = 0$, hen $\psi$ is continuous on $[0,\infty)$.

     \item[(b)] If $f \in L_{loc}^{3/2}(\R_+; \C_-^m)$, $K \in L_{loc}^3(\R_+; \R^{m\times m})$, and there exists $\alpha \in (0,1)$ such that $K$ is locally $\alpha$-H\"older continuous and for each $T > 0$ we can find a constant $C(T) > 0$ satisfying for $i=1,\dots, m$
     \[
      \int_0^T |K_i(t+h) - K_i(t)|^3dt + \int_0^h |K_i(t)|^3 dt \leq C(T)h^{3\alpha}, \qquad h \in [0,1],
     \]
     then $\psi(\cdot, u,f)$ is locally $\alpha$-H\"older continuous on $(0,\infty)$.
 \end{enumerate}
\end{Theorem}
\begin{proof}
 (a) Let us first show that $\psi(\cdot, u,f)$ has a continuous representative on $(0,\infty)$. Using the representation \eqref{eq: riccati}, it suffices to show that each term on the right-hand side of \eqref{eq: riccati} is continuous on $(0,\infty)$. The first term $Ku$ is continuous by assumption. The convolution $K \ast f$ is continuous on $[0,\infty)$ by Remark \ref{rem: continuity} due to $K \in L_{loc}^3$ and $f \in L_{loc}^{3/2}$. Similarly, the continuity of the function $t \longmapsto \int_0^t K(t-s)R(\psi(s,u,f))ds$ follows from Remark \ref{rem: continuity} if we can show that $R(\psi(\cdot, u,f)) \in L_{loc}^{3/2}$.
 Since $|R(\psi)| \leq C\left(1 + |\psi|^2 \right)$ holds for some constant $C > 0$, the fact that $R(\psi(\cdot, u,f)) \in L_{loc}^{3/2}$ is a consequence of $K \in L_{loc}^3$ and Lemma \ref{lemma: local integrability psi} applied for $p = 3$. Hence $\psi(\cdot,u,f)$ has a continuous representative on $(0,\infty)$. If $u = 0$, then $\psi$ is continuous on $[0,\infty)$ by above reasoning and since $Ku = 0$.

 (b) Let $\psi$ be the continuous representative obtained in part (a). We show that this representative is locally $\alpha$-H\"older continuous. Fix $T > 0$ and let $s,t \in [0,T]$ be such that $0 \leq t-s \leq 1$. Then using \eqref{eq: riccati} we find that $|\psi(t,u,f) - \psi(s,u,f)| \leq I_1 + I_2 + I_3$ with
 \begin{align*}
  I_1 &= |K(t)u - K(s)u|
  \\ I_2 &= \left| \int_0^t K(t-r)f(r)dr - \int_0^s K(s-r)f(r)dr\right|
  \\ I_3 &= \left| \int_0^t K(t-r)R(\psi(r,u,f))dr - \int_0^s K(s-r)R(\psi(r,u,f))dr \right|.
 \end{align*}
 Here and below we let $C > 0$ be a generic constant which is independent of $s,t$.
 The first term $I_1$ can be estimated by the $\alpha$-H\"older continuity of $K$, i.e., $I_1 \leq C(t-s)^{\alpha}|u|$. For the second term we obtain
 \begin{align*}
     I_2 &\leq \int_0^s |(K(t-r) - K(s-r))f(r)| dr
     + \int_s^t |K(t-r)f(r)| dr
     \\ &\leq  \left( \int_0^s \|K(t-r) - K(s-r)\|_2^3 dr\right)^{1/3} \left( \int_0^s |f(r)|^{3/2} dr\right)^{2/3}
     \\ & \ \ \ + \left( \int_s^t \|K(t-r)\|_2^3 dr \right)^{1/3}\left( \int_0^s |f(r)|^{3/2} dr\right)^{2/3}
     \\ &\leq C (t-s)^{\alpha}.
 \end{align*}
 Finally, the last term $I_3$ can be treated in the same way as in the case of $I_2$, since $R(\psi(\cdot,u,f)) \in L_{loc}^{3/2}$ as already proved in part (a).
 Thus we have shown that
 $|\psi(t,u,f) - \psi(s,u,f)| \leq C(t-s)^{\alpha}$,
 which proves the assertion.
\end{proof}

\begin{Example}
 The kernel $K(t) = t^{H - 1/2}/\Gamma(H + 1/2)I_m$ with $H \in (0,1/2)$ belongs to $L_{loc}^p$ if and only if $p < \frac{1}{\frac{1}{2} - H}$. In particular, for $p = 3$ we obtain $H > 1/6$.
\end{Example}

For $H \in (0, 1/6)$ we will show that $\psi(\cdot,u,f)$ has at least some fractional Sobolev regularity. Indeed, we prove this for a general class of kernels $K$.

Given $p \geq 2$ and $\eta \in (0,1)$ let $W^{\eta,p}([0,T])$ be the Banach space of equivalence classes of functions $g: [0,T] \longrightarrow \C^m$ with finite norm
\[
 \| g \|_{W^{\eta,p}([0,T])} = \left( \int_0^T |g(t)|^p dt + \int_0^T \int_0^T \frac{|g(t) - g(s)|^p}{|t-s|^{1 + \eta p}}ds dt \right)^{1/p}.
\]
Finally, define
\[
 [K]_{\eta,p,T} = \left(\int_0^T t^{-\eta p}\|K(t)\|_2^p dt
 + \int_0^T\int_0^T \frac{\|K(t) - K(s)\|_2^p}{|t-s|^{1+\eta p}}dsdt \right)^{1/p}.
\]

\begin{Example}
 The kernel $K(t) = \frac{t^{H-1/2}}{\Gamma(H + 1/2)}e^{-\lambda t}I_m$ with $H \in (0,1/2)$ and $\lambda \geq 0$ satisfies $[K]_{\eta,p,T} < \infty$ for each $T > 0$, $p = 2$, and $\eta \in (0, H)$, see \cite{ACLP19}.
\end{Example}

The following is our second regularity result in $t$ for the solution of \eqref{eq: riccati}.
\begin{Theorem}\label{Theorem: sobolev regularity psi}
 Let $(b,\beta, \sigma, K)$ be admissible parameters and let $\psi$ be the unique solution of \eqref{eq: riccati} in $L_{loc}^2(\R_+; \C_-^{m})$, where $u \in \C_-^m$ and $f \in L_{loc}^1(\R_+; \C_-^{m})$.
 Suppose there exist $p \geq 2$, $\eta \in (0,1)$ and $T > 0$
 with $[K]_{\eta,p,T} < \infty$. Then
 \begin{align*}
  &\ \| \psi(\cdot, u,f) \|_{W^{\eta,p}([0,T])}
  \leq \| \psi(\cdot, u,f) \|_{L^p([0,T])}
  \\ &\qquad \qquad +  C(1 + [K]_{\eta,p, T})\left( 1 + |u| + \| f \|_{L^1([0,T])} +  \| \psi(\cdot,u,f) \|_{L^2([0,T])}^{2} \right),
 \end{align*}
 where the constant $C$ only depends on $T,p,m,\beta,\sigma$.
\end{Theorem}
\begin{proof}
 Here and below we let $C$ be a generic constant independent of $u$ and $f$. Let $I_1,I_2,I_3$ be the same as in the proof of Theorem \ref{Theorem regularity 1}.
 Then we obtain $|\psi(t,u,f) - \psi(s,u,f)|^p \leq C I_1^p+C I_2^p+C I_3^p$
 for $0 \leq s < t \leq T$.
 It is easy to see that
 \begin{equation} \label{esti I1}
   I_1^p \leq  \|K(t) - K(s)\|_2^p |u|^p.
 \end{equation}
 Let $g(r):=|R(\psi(r,u,f)|$. Since $|R(u)| \leq C(1 + |u|^2)$, it holds
 \begin{equation}\label{eq integral R psi}
  \int_0^T g(r)dr \leq C \left( 1  + \| \psi \|_{L^2([0,T])}^2 \right).
 \end{equation}
 So
 \begin{align*}
  I_3^p &\leq C \left(\int_s^t \| K(t-r)\|_2 g(r) dr \right)^p
  \\ &\ \ \ + C \left( \int_0^s \| K(t-r) - K(s-r)\|_2 g(r)dr \right)^p
  \\ &\leq
     C \left( \int_s^t g(r)dr \right)^{p-1} \left(\int_s^t \| K(t-r)\|_2^p g(r) dr \right)
  \\ &\ \ \ + C \left( \int_0^s g(r)dr \right)^{p-1}\left( \int_0^s \| K(t-r) - K(s-r)\|_2^p g(r)dr \right)
  \\ &\leq  C\left( 1 +  \| \psi \|_{L^2([0,T])}^{2p-2} \right)\left(\int_s^t \| K(t-r)\|_2^p g(r) dr \right)
  \\ & \ \ \ + C \left( 1 +  \| \psi \|_{L^2([0,T])}^{2p-2} \right)\left( \int_0^s \| K(t-r) - K(s-r)\|_2^p g(r)dr \right),
 \end{align*}
 where we have used (\ref{eq integral R psi}) and  the elementary inequality
  \begin{align*}
  \left(\int_a^b h(r)g(r)dr \right)^p \leq \left( \int_a^b g(r)dr \right)^{p-1} \int_a^b h(r)^p g(r)dr
 \end{align*}
 for $h \geq 0$ and $0 \leq a < b$. Thus we obtain
 \begin{align}
     &\ \int_0^T \int_0^T \frac{I_3^p}{|t-s|^{1 + \eta p}}dsdt \notag
     \\ &= 2\int_0^T \int_0^t \frac{I_3^p}{|t-s|^{1 + \eta p}}dsdt \notag
      \\ &\leq  C \left( 1 +  \| \psi \|_{L^2([0,T])}^{2p-2} \right)\int_0^T \int_0^t \int_s^t \frac{\|K(t-r)\|_2^p}{|t-s|^{1 + \eta p}}g(r)dr ds dt \notag
      \\ &\ \ \ + C \left( 1  +  \| \psi \|_{L^2([0,T])}^{2p-2} \right) \int_0^T \int_0^t \int_0^s \frac{\|K(t-r) - K(s-r)\|_2^p}{|t-s|^{1 + \eta p}}g(r)dr ds dt \notag
      \\ &\leq   C\left( 1  +  \| \psi \|_{L^2([0,T])}^{2p-2} \right)[K]_{\eta,p,T}\left( \int_0^T g(r) dr \right)\label{before eq last last}
      \\ &\leq C(1 + [K]_{\eta,p, T}^p)\left( 1  +  \| \psi \|_{L^2([0,T])}^{2p} \right), \label{eq last last}
 \end{align}
 where in \eqref{before eq last last} we have used Fubini's theorem to get
  \begin{align*}
     &\ \int_0^T \int_0^t \int_s^t \frac{\|K(t-r)\|_2^p}{|t-s|^{1 + \eta p}}g(r)dr ds dt \\
     &  =\int_0^T \int_r^T \int_0^r \frac{\|K(t-r)\|_2^p}{|t-s|^{1 + \eta p}}g(r) ds dtdr \\
     &  \le \frac{1}{\eta p}\int_0^T \int_r^T  \frac{\|K(t-r)\|_2^p}{|t-r|^{ \eta p}}g(r)  dtdr \\
     & \le (\eta p)^{-1} [K]_{\eta,p,T}\left( \int_0^T g(r) dr \right).
 \end{align*}
 Repeating the above arguments for $I_2$, we obtain
  \begin{align} \label{esti I2}
     \int_0^T \int_0^T \frac{I_2^p}{|t-s|^{1 + \eta p}}dsdt \leq  C(1 + [K]_{\eta,p, T}^p)\left( 1 +  \| f \|_{L^1([0,T])}^{p}  \right).
 \end{align}
Summarizing \eqref{esti I1}, \eqref{eq last last} and \eqref{esti I2},  we obtain
 \begin{align*}
     &\ \int_0^T \int_0^T \frac{|\psi(t,u,f) - \psi(s,u,f)|^p}{|t-s|^{1 + \eta p}}dsdt
      \\ &\leq C |u|^p  \int_0^T \int_0^T \frac{\|K(t) - K(s)\|_2^p}{|t-s|^{1 + \eta p}}ds dt
            \\ & \ \ \ + C(1 + [K]_{\eta,p, T}^p)\left( 1 +  \| f \|_{L^1([0,T])}^{p} +\| \psi \|_{L^2([0,T])}^{2p} \right)
      \\ &\leq C(1 + [K]_{\eta,p, T}^p)\left( 1 + |u|^p + \| f \|_{L^1([0,T])}^{p} +  \| \psi \|_{L^2([0,T])}^{2p} \right).
 \end{align*}
 In view of Lemma \ref{lemma: local integrability psi}, the assertion is proved.
\end{proof}

\subsection{Extension to measure convolutions}

Let $\mathcal{M}_{lf}^- \subset \mathcal{M}_{lf}$ be the subset of $\C^m$-valued set functions $\mu \in \mathcal{M}_{lf}$  which satisfy $\mathrm{Re}(\mu) \leq 0$.
Below we extend the affine formula from $(u,f) \in \C_-^m \times L_{loc}^1(\R_+; \C_-^m)$ to $(u,\mu) \in \C_-^m \times \mathcal{M}_{lf}^-$. The latter one provides us the key tool to explicitly identify the finite dimensional distributions of the stationary process via a Volterra Riccati-type equation (see Section 5).
\begin{Lemma}\label{lemma: admissible mu approximation}
 For each $\mu \in \mathcal{M}_{lf}^-$ there exists $(f_n)_{n \geq 1} \subset L_{loc}^1(\R_+; \C_-^m)$ such that
 \begin{enumerate}
     \item[(i)] $\| f_n\|_{L^1([0,T])} \leq |\mu|([0,T])$ for each $T > 0$;

     \item[(ii)] For each $T > 0$, $p \geq 1$ and $g \in L^p([0,T]; \C^m)$ one has $
      g^{\top} \ast f_n \to g^{\top} \ast \mu$ in $ \ L^p([0,T])$;
     \item[(iii)] For each $T > 0$ and each $g \in C([0,T]; \C^m)$ with $g(0)=0$ one has
     \[
      \lim_{n \to \infty}\int_0^t \langle g(t-s),f_n(s)\rangle ds = \int_{[0,t]} \langle g(t-s),\mu(ds)\rangle
     \]
     pointwise for each $t \in [0,T]$.
 \end{enumerate}
\end{Lemma}
\begin{proof}
 Let $\rho_n(t) = ne^{-n t}$, $t\ge 0$ and define $f_n(t) = \int_{[0,t]} \rho_n(t-s) \mu(ds)$. Then $\mathrm{Re}(f_n) \leq 0$ and $\| f_n \|_{L^1([0,T])} \leq |\mu|([0,T])$. Classical results from Fourier analysis (see \cite[2.16 Theorem]{MR1817225}) shows that
 $\rho_n \ast h \to h$ in $L^p$ if $h \in L^p$. The assertion (ii) now follows from
 $g_i \ast f_{n,i} = g_i \ast \rho_n \ast \mu_i = \rho_n \ast ( g_i \ast \mu_i )$, $i=1,\ldots,m$.
 Let's turn to (iii). First, for $t=0$, the desired convergence is true due to $g(0)=0$. Suppose $t>0$. We note that
 \begin{align*}
     \int_0^t  g_i(t-s) f_{n,i}(s)ds
     &= (g_i \ast \rho_n \ast \mu_i)(t)
     \\ &= \int_{[0,t]} (g_i \ast \rho_n) (t-s)\mu_i(ds)
     \\ &= \int_{[0,t)} g_i \ast \rho_n (t-s)\mu_i(ds),
 \end{align*}
 where we have used $(g_i \ast \rho_n)(0)\mu(\{t\}) = 0$. By dominated convergence,  the right-hand side of the last equality converges to $\int_{[0,t)} g_i  (t-s)\mu_i(ds)$, since $g_i \ast \rho_n (t-s) \to g_i(t-s)$ for each $s\in [0,t)$ as $n\to\infty$. Hence
 \begin{align*}
   \lim_{n\to\infty}  \int_0^t  g_i(t-s) f_{n,i}(s)ds = \int_{[0,t)} g_i  (t-s)\mu_i(ds) = \int_{[0,t]} g_i  (t-s)\mu_i(ds),
 \end{align*}
 which proves the assertion (iii).
\end{proof}

\begin{Remark}
 If $\mu$ is non-atomic, i.e., $\mu(\{t\}) = 0$ holds for each $t \geq 0$, then the additional assumption $g(0) = 0$ can be omitted.
\end{Remark}

The next result extends \eqref{eq: riccati} in the sense that one can replace $f$ by $\mu \in \mathcal{M}_{lf}^-$.
\begin{Theorem}\label{Theorem: riccati extension}
 Let $(b,\beta, \sigma, K)$ be admissible parameters and  suppose there exist $p \geq 2$ and $\eta \in (0,1)$ such that
 $[K]_{\eta,p,T} < \infty$  for each $T > 0$. Then the following assertions hold:
 \begin{enumerate}
     \item[(a)] For each $\mu \in \mathcal{M}_{lf}^-$ there exists a unique $L_{loc}^2(\R_+;\C_-^m)$-valued solution
 \begin{align}\label{eq: riccati extension}
  \psi(t,\mu) = \int_{[0,t)} K(t-s)\mu(ds) + \int_{0}^{t}K(t-s)R(\psi(s,\mu))ds,\quad t\ge 0.
 \end{align}

 \item[(b)] For each $q \in [1,p]$ this unique solution satisfies
  \begin{align*}
  &\ \| \psi(\cdot, \mu) \|_{L^q([0,T])} \leq 2|\mu|([0,T])
  \| E_{\beta^{\top}}\|_{L^q([0,T])}
  \\ &\qquad \qquad + \left( \sum_{i=1}^m \frac{\sigma_i^2}{2} \right)|\mu|([0,T])^2\| E_{\beta^{\top}}\|_{L^q([0,T])} \| E_{\beta^{\top}}\|_{L^2([0,T])}^2
 \end{align*}
 and if $\mathrm{Im}(\mu)= 0$, then even
 \[
  \| \psi(\cdot,\mu)\|_{L^q([0,T])} \leq |\mu|([0,T]) \| E_{\beta^{\top}} \|_{L^q([0,T])}.
 \]
 Finally, there exists a constant $C$ independent of $\mu$ such that
 \begin{align*}
 &\ \| \psi(\cdot, \mu)\|_{W^{\eta,p}([0,T])} \leq \|\psi(\cdot,\mu)\|_{L^p([0,T])}
  \\ &\qquad \qquad + C(1+ [K]_{\eta,p, T}) \left( 1 + |\mu|([0,T]) + \| \psi(\cdot,\mu)\|_{L^2([0,T])}^2 \right).
 \end{align*}

  \item[(c)] If $p \geq 3$, then the function $\psi(\cdot,\mu)$ is continuous at each $t_0 \geq 0$ for which the convolution
     $K \ast \mu (\cdot)$ is continuous at $t_0$.
 \end{enumerate}
\end{Theorem}
\begin{proof}
 Let $(f_n)_{n \geq 1} \subset L_{loc}^1(\R_+; \C_-^m)$ be a sequence of functions as given in Lemma \ref{lemma: admissible mu approximation}. Let $\psi_n = \psi(\cdot, 0, f_n)$ be the sequence of unique solutions of \eqref{eq: riccati}. Fix $T > 0$. Then using Lemma \ref{lemma: local integrability psi} and Lemma \ref{lemma: admissible mu approximation}.(i) we obtain for each $q \in [1,p]$,
  \begin{align*}
   \| \psi(\cdot, 0, f_n) \|_{L^q([0,T])} &\leq 2|\mu|([0,T])
  \| E_{\beta^{\top}}\|_{L^q([0,T])}
  \\ & \ \ \ + \left( \sum_{i=1}^m \frac{\sigma_i^2}{2} \right)|\mu|([0,T])^2\| E_{\beta^{\top}}\|_{L^q([0,T])} \| E_{\beta^{\top}}\|_{L^2([0,T])}^2,
 \end{align*}
 and if $\mathrm{Im}(\mu) = 0$, then
 \[
  \| \psi(\cdot,0, f_n)\|_{L^q([0,T])} \leq |\mu|([0,T]) \| E_{\beta^{\top}} \|_{L^q([0,T])}.
 \]
 Hence Theorem \ref{Theorem: sobolev regularity psi} combined with Lemma \ref{lemma: admissible mu approximation}.(i) implies that
 \begin{align*}
  &\ \| \psi(\cdot, 0,f_n)\|_{W^{\eta,p}([0,T])} \leq \|\psi(\cdot,0,f_n)\|_{L^p([0,T])}
  \\ &\qquad \qquad + C(1+ [K]_{\eta,p, T}) \left( 1 + |\mu|([0,T]) + \| \psi(\cdot,0,f_n)\|_{L^2([0,T])}^2 \right).
 \end{align*}
 In view of the $L^q$-estimates on $\psi(\cdot,0,f_n)$ and Remark \ref{integrability of RB and EB}, the right-hand side is bounded in $n$.
 Since the ball $\{g \in L^p([0,T]) \, : \, \|g \|_{W^{\eta,p}([0,T])} \leq R \}$ with $R > 0$ are relatively compact in $L^p([0,T]; \C^m)$ (see \cite[Theorem 2.1]{MR1339739}), we find a subsequence $(f_{n_k})_{k \geq 1}$ such that $\psi(\cdot, 0,f_{n_k}) \longrightarrow \psi$ in $L^p([0,T];\C^m)$. Further, we can choose a subsubsequence, still denoted  by $(f_{n_k})$, such that $\psi(\cdot, 0,f_{n_k})$ converges almost surely to $\psi$ on $[0,T]$.
 Taking the limit $k \to \infty$ and using the Lemma of Fatou proves the estimates from part (b).

 Next we show that $\psi = \psi(\cdot, \mu)$ is a solution of \eqref{eq: riccati extension} on $[0,T]$. Since $\psi_{n_k} \to \psi$ and $K\ast f_{n_k} \longrightarrow K\ast \mu$ in $L^p([0,T]; \C^m)$, it suffices to show that $K \ast R(\psi_{n_k})
  \longrightarrow K \ast R(\psi)$ holds in $L^p([0,T]; \C^m)$.  For this purpose we first use Young's inequality, then
 \begin{equation}\label{esti: R}
   |R(u) - R(v)| \leq C(1 + |v| + |u|)|u-v|,
 \end{equation}
 and finally the Cauchy-Schwartz inequality to find that
 \begin{align*}
  &\ \| K\ast R(\psi_{n_k}) - K \ast R(\psi) \|_{L^p([0,T])}
   \\ &\leq C \| K \|_{L^p([0,T])} \int_0^T (1 + |\psi_{n_k}(s)| + |\psi(s)|)|\psi_{n_k}(s) - \psi(s)|ds
   \\ &\leq C \| K \|_{L^p([0,T])} \left( 1 + \| \psi_{n_k}\|_{L^2([0,T])} + \|\psi\|_{L^2([0,T])} \right) \|\psi_{n_k} - \psi \|_{L^2([0,T])}.
 \end{align*}
 Since the right-hand side converges to zero, we find that
 $\psi$ is a global solution of \eqref{eq: riccati extension}. Noting that \eqref{esti: R} holds and that $K \ast \mu \in L_{loc}^2(\R_+; \C^m)$,  \cite[Theorem B.1]{MR4019885} implies that this equation has a unique maximal solution. Since $\psi$ is a global solution, the unique maximal solution is defined on all $\R_+$ and coincides with $\psi$. This proves part (a).

 To prove part (c),
 in view of \eqref{eq: riccati extension}, it suffices to show that $K \ast R(\psi)$ is continuous on $\R_+$. The latter one is true, if $K \in L_{loc}^3$ and $R(\psi) \in L_{loc}^{3/2}$, which holds true due to $|R(\psi)|\leq C (1+|\psi|^2)$ and $\psi \in L_{loc}^{3/2}$. This proves part (c).
\end{proof}

Finally, we extend the exponential-affine transformation formula.
\begin{Corollary}\label{Theorem: affine riccati extension}
 Let $(b,\beta, \sigma, K)$ be admissible parameters and  suppose there exist $p \geq 2$ and $\eta \in (0,1)$ such that
 $[K]_{\eta,p,T} < \infty$ for each $T > 0$. Then
 \begin{align*}
  &\ \mathbb{E}\left[\mathrm{e}^{\int_{[0,t]} \langle X_{t-s}, \mu(ds) \rangle }\right]
  \\ &\ \ \ = \exp\left\{ \langle x_0, \mu([0,t]) \rangle + \int_0^t \langle x_0, R(\psi(s,\mu)) \rangle ds + \int_0^t \langle b, \psi(s,\mu) \rangle ds \right\}
     \\ &\ \ \ = \exp\left\{ \int_{[0,t]} \langle \E[X_{t-s}], \mu(ds)\rangle + \sum_{i=1}^m \frac{\sigma_i^2}{2}\int_0^t \E[X_{i,t-s}] \psi_{i}(s,\mu)^2 ds \right\}
 \end{align*}
 hold for each $\mu \in \mathcal{M}_{lf}^-$, where $\psi$ denotes the unique solution of \eqref{eq: riccati extension}.
\end{Corollary}
\begin{proof}
 For the first equality, we let $f_{n_k},\psi_{n_k}$ be the same as in the proof of Theorem \ref{Theorem: riccati extension}. Then
 \begin{align}\label{new eq - x0}
 \mathbb{E}\left[\mathrm{e}^{\int_0^t \langle X_s-x_0,f_{n_k}(t-s) \rangle ds}\right]
 = \exp\left\{  \int_0^t \langle x_0, R(\psi_{n_k}(s)) \rangle ds + \int_0^t \langle b, \psi_{n_k}(s) \rangle ds \right\}.
\end{align}
 Using Lemma \ref{lemma: admissible mu approximation}.(iii) for $g(s) = X_s-x_0$ gives $\int_0^t \langle X_s-x_0, f_{n_k}(t-s) \rangle ds \longrightarrow \int_{[0,t]} \langle X_{t-s}-x_0, \mu(ds) \rangle$ pointwise. Dominated convergence gives
 \begin{align}\label{eq:02}
  \lim_{n \to \infty}\mathbb{E}\left[\mathrm{e}^{\int_0^t \langle X_{t-s}-x_0, f_{n_k}(s) \rangle ds}\right] = \mathbb{E}\left[\mathrm{e}^{\int_{[0,t]} \langle X_{t-s}-x_0, \mu(ds) \rangle}\right], \qquad \forall t \geq 0.
 \end{align}
 For the two terms appearing on the right-hand side of \eqref{new eq - x0} we have
 \[
  \left| \int_0^t \langle b, \psi_{n_k}(s) \rangle ds - \int_0^t \langle b, \psi(s)\rangle ds \right| \leq |b|\sqrt{T} \| \psi_{n_k} - \psi \|_{L^2([0,T])}
 \]
 and likewise, using \eqref{esti: R} we obtain
 \begin{align*}
     &\ \left| \int_0^t \langle x_0, R(\psi_{n_k}(s))\rangle ds - \int_0^t \langle x_0, R(\psi(s))\rangle ds \right|
     \\ &\leq C|x_0| \left( \sqrt{T} + \sup_{k \geq 1}\| \psi_{n_k}\|_{L^2([0,T])} + \|\psi\|_{L^2([0,T])} \right) \| \psi_{n_k} - \psi\|_{L^2([0,T])}.
 \end{align*}
 Hence passing to the limit $k\to \infty$ in \eqref{new eq - x0} proves the first identity. For the second identity, first observe that
   \begin{align*}
 & \mathbb{E}\left[\mathrm{e}^{\int_0^t \langle X_s-x_0,f_{n_k}(t-s) \rangle ds}\right] \\
 &\qquad = \exp\left\{ \int_0^t \langle \E[X_{t-s}]-x_0,f_{n_k}(s)\rangle ds + \sum_{i=1}^m \frac{\sigma_i^2}{2} \int_0^t \E[X_{i,t-s}] \psi_{i,n}(s)^2 ds \right\}
\end{align*}
holds by \cite[equations (4.5),(4.7)]{MR4019885}.
 In view of \eqref{eq:02} it suffices to show that the right-hand side is converges to the desired limit as $n\to \infty$. Since $g(r) = \E[X_{r}]-x_0$ is continuous (see \eqref{lemma: first moment} below) and $g(0)=0$, by Lemma \ref{lemma: admissible mu approximation}.(iii), we find
 $\lim_{k\to\infty}\int_0^t \langle \E[X_{t-s}]-x_0,f_{n_k}(s)\rangle ds
 = \int_{[0,t]} \langle \E[X_{t-s}]-x_0, \mu(ds)\rangle$. For the second term we note that
 \begin{align*}
     &\ \left| \int_0^t \E[X_{i,t-s}] \psi_{i,n}(s)^2 ds - \int_0^t \E[X_{i,t-s}] \psi_{i}(s)^2 ds\right|
    \\ &\leq \sup_{s \in [0,t]}\E[X_{i,s}] \int_0^t (|\psi_{i,n}(s)| + |\psi_{i}(s)|)| \psi_{i,n}(s) - \psi_{i}(s)|ds
     \\ &\leq \sup_{s \in [0,t]}\E[X_{i,s}] \left( \sup_{k \geq 1}\| \psi_{n_k}\|_{L^2([0,T])} + \| \psi\|_{L^2([0,T])} \right) \| \psi_{n_k} - \psi \|_{L^2([0,T])},
 \end{align*}
 which tends to zero as $k \to \infty$. Here we used the fact that $\sup_{s \in [0,t]}\E[X_{i,s}]<\infty$ due to \cite{MR4019885}.
\end{proof}

\subsection{Stability in the admissible parameters}

As a particular application of our results, we show that the Volterra square-root process depends continuously on the admissible parameters $(b,\beta,\sigma,K)$.
\begin{Theorem}\label{Theorem: stability psi}
 Let $(b,\beta,\sigma,K), (b,\beta^{(n)},\sigma^{(n)},K^{(n)})$ be admissible parameters with the properties
 (i) $\beta^{(n)} \longrightarrow \beta$;
 (ii) $\sigma^{(n)} \longrightarrow \sigma$;
 (iii) There exists $p \geq 2$ such that $\|K^{(n)} - K \|_{L^p([0,T])} \longrightarrow 0$ for each $T > 0$;
 (iv) There exists $\eta \in (0,1)$ with $[K]_{\eta,p,T} + \sup_{n \geq 1}[K^{(n)}]_{\eta,p,T} < \infty$ for each $T > 0$. For $\mu, \mu^{(n)} \in \mathcal{M}_{lf}^-$ let $\psi$ and $\psi_n$ be the corresponding unique solutions of \eqref{eq: riccati extension}. Suppose that $|\mu^{(n)} - \mu|([0,T]) \longrightarrow 0$ for each $T > 0$. Then
 \[
  \lim_{n \to \infty}\| \psi_n - \psi \|_{L^p([0,T])} = 0.
 \]
\end{Theorem}
\begin{proof}
 Since $\beta^{(n)},\sigma^{(n)}$ converge, they are bounded. Similarly, we have $\sup_{n\ge 1}\| K^{(n)}\|_{L^2([0,t])}<\infty$ for each $T>0$.
 Noting that $K^{(n)}(\beta^{(n)})^{\top}\to K\beta^{\top}$  in $L^p([0,T])$  for each $T>0$, by \cite[Theorem 2.3.1]{MR1050319}, we also obtain $\sup_{n\ge1}\|R_{(\beta^{(n)})^{\top}}^{(n)}\|_{L^1([0,T])}<\infty$. In view of the definition in \eqref{eq: def Ebeta}, we thus get $\sup_{n\ge1}\|E_{(\beta^{(n)})^{\top}}^{(n)}\|_{L^p([0,T])}<\infty$ by Young's inequality. Moreover, the particular form of the inequalities in Theorem \ref{Theorem: riccati extension}.(b) as well as the constant $C$ imply $\sup_{n \geq 1}\| \psi_n \|_{L^p([0,T])} < \infty$ and subsequently
 $\sup_{n \geq 1}\| \psi_n \|_{W^{\eta,p}([0,T])} < \infty$.
 Hence $(\psi_n)_{n \in \N} \subset L^p([0,T];\C_-^m)$ is relatively compact. Let $\widetilde{\psi}$ be the limit for some subsequence $\psi_{n_k}$. If we can show that $\tilde{\psi}$ is a solution of \eqref{eq: riccati extension}, by uniqueness, we must have $\psi = \tilde{\psi}$ and thus the convergence $\| \psi_n - \psi \|_{L^p([0,T])} = 0$ as $n\to\infty$, since the convergent subsequence is arbitrarily chosen.

 Next we show that $\tilde{\psi}$ is a solution of \eqref{eq: riccati extension}, i.e.,  $\tilde{\psi} = K \ast \mu + K \ast R(\tilde{\psi})$. Noting that $\psi_n = K^{(n)} \ast \mu^{(n)} +K^{(n)} \ast R^{(n)}(\psi_n)$ with $R_i^{(n)}(u) = \langle u, \beta^{i,(n)}\rangle + \frac{(\sigma_i^{(n)})^2}{2}u_i^2$ and $\beta^{i,(n)} = (\beta_{1i}^{(n)}, \dots, \beta_{mi}^{(n)})^{\top}$, it suffices to show that
 \[
   \lim_{n \to \infty}\| K^{(n)} \ast \mu^{(n)} + K^{(n)} \ast R^{(n)}(\psi_n)
     - K \ast \mu -K \ast R(\widetilde{\psi})\|_{L^p([0,T])} = 0.
 \] 
 Using the properties (i) -- (iv) combined to similar estimates to the proofs of previous sections, it is not difficult to see that this convergence is satisfied. The details are left for the reader. 
\end{proof}

Consequently, we can now prove that the law of the Volterra square-root process depends continuously on the parameters.
\begin{Corollary}
 Let $(b,\beta,\sigma,K), (b^{(n)},\beta^{(n)},\sigma^{(n)},K^{(n)})$ be admissible parameters with the properties (i) -- (iv) from Theorem \ref{Theorem: stability psi}, and $b^{(n)} \longrightarrow b$.
 Let $X$ and $X^n$ be the Volterra square-root processes with admissible parameters $(b,\beta,\sigma, K)$ and $(b^{(n)}, \beta^{(n)}, \sigma^{(n)}, K^{(n)})$ starting from the same initial state. Then the law of $X^{(n)}$ on $C(\R_+; \R_+^m)$ converges weakly to that of $X$.
\end{Corollary}
\begin{proof}
 For $\mu \in \mathcal{M}_{lf}^-$ let $\psi$ and $\psi_n$ be the corresponding unique solutions of \eqref{eq: riccati extension}. Let $R^{(n)}$ be the same as in the proof of Theorem \ref{Theorem: stability psi}, where we implicitly showed that $R^{(n)}(\psi_n) \to R(\psi)$ in $L^1([0,T])$. Then using $\psi_n \to \psi$ in $L^p([0,T])$ and the first identity in Corollary \ref{Theorem: affine riccati extension}, we find that
 \[
  \lim_{n \to \infty}\E[ e^{\int_{[0,T]}\langle X^n_{t-s}, \mu(ds)\rangle}]
  = \E[ e^{\int_{[0,T]}\langle X_{t-s}, \mu(ds)\rangle}].
 \]
 In particular, letting $\mu(ds) = \sum_{j=1}^n u_j \delta_{t_j}(ds)$ with $u_1,\dots, u_n \in \C_-^m$, and $0 \leq t_1 < \dots < t_n$ shows that the finite dimensional distributions of $X^n$ are convergent to those of $X$. Arguing as in the proof of \cite[Lemma A.1]{MR4019885}, we also know that $X^n$ is tight in $C(\R_+; \R_+^m)$. So for any subsequence of $X^n$, it has a subsequence converging in law to $X$. This proves that the law of $X^{(n)}$ on $C(\R_+; \R_+^m)$ converges weakly to that of $X$.
\end{proof}

\subsection{Differentiability in the initial condition}

In this section we study continuity and differentiability of $\psi$ with respect to the initial condition $\mu$.
\begin{Theorem}\label{Theorem: continuity in mu}
 Let $(b,\beta, \sigma, K)$ be admissible parameters and  suppose there exist $p \geq 2$ and $\eta \in (0,1)$ such that $[K]_{\eta,p,T} < \infty$ for each $T > 0$. Then for each pair of $\mu,\nu \in \mathcal{M}_{lf}^-$, there exists a constant $C(T,p) > 0$ such that
 \[
  \| \psi(\cdot, \mu + \e \nu) - \psi(\cdot, \mu) \|_{L^p([0,T])} \leq C(T,p)\e.
 \]
\end{Theorem}
\begin{proof}
 Note that $\Delta_{\e}(t) := \psi(t, \mu + \e \nu) - \psi(t,\mu)$ satisfies
 \begin{align*}
     \Delta_{\e}(t) &= \e (K \ast \nu)(t) + \int_0^t K(t-s)(R(\psi(s, \mu + \e \nu)) - R(\psi(s, \mu)))ds.
 \end{align*}
 Hence we obtain from \eqref{esti: R}
 \begin{align*}
    &\  |\Delta_{\e}(t)|^2 \leq 2\e^2 |K\ast \nu(t)|^2
     \\ & \ \ \ + 2C^2 \left(\int_0^t (1 + |\psi(s,\mu + \e \nu)| + |\psi(s,\mu)|)^2 ds \right) \int_0^t \| K(t-s)\|_2^2 |\Delta_{\e}(s)|^2 ds
     \\ &\leq 2\e^2 |K\ast \nu(t)|^2 + C(\mu,\nu) \int_0^t \|K(t-s)\|_2^2 |\Delta_{\e}(s)|^2 ds,
 \end{align*}
 where
 \[
  C(\mu, \nu) = 6C^2 \left( 1 + \sup_{\e \in (0,1)}\| \psi(\cdot, \mu + \e\nu) \|_{L^2([0,T])}^2 + \| \psi(\cdot, \mu)\|_{L^2([0,T])}^2 \right)
 \]
 is finite due to Theorem \ref{Theorem: riccati extension}.
 Now let $\widetilde{R} \in L_{loc}^{p/2}([0,T]; \R)$ be the resolvent of the second kind of $-C(\mu,\nu)\|\widetilde{K}(\cdot)\|_2^2 \in L_{loc}^{p/2}([0,T]; \R)$. Using \cite[Proposition 9.8.1]{MR1050319} we find $\widetilde{R} \leq 0$. Hence a Volterra analogue of the Gronwall inequality (see \cite[Theorem A.2]{ACLP19}) gives
 \begin{align*}
  |\Delta_{\e}(t)|^2 &\leq 2\e^2 |K\ast \nu(t)|^2 + 2 \e^2 \int_0^t |\widetilde{R}(t-s)||K\ast \nu(s)|^2 ds.
 \end{align*}
 This gives
 \begin{align*}
     &\ \| \Delta_{\e}\|_{L^p([0,T])}
     \\ &\leq 2\e \| K \ast \nu\|_{L^p([0,T])} + 2\e \| |\widetilde{R}| \ast |K\ast \nu|^{2}\|_{L^{p/2}([0,T])}^{1/2}
     \\ &\leq 2\e \| K\|_{L^p([0,T])}|\nu|([0,T]) + 2\e \| \widetilde{R}\|_{L^{p/2}([0,T])}^{1/2} \| K\|_{L^2([0,T])}|\nu|([0,T])
 \end{align*}
 and hence proves the assertion.
\end{proof}

Next we prove differentiability in $\mu$.
\begin{Theorem}
 Let $(b,\beta, \sigma, K)$ be admissible parameters and  suppose there exist $p \geq 2$ and $\eta \in (0,1)$ such that $[K]_{\eta,p,T} < \infty$ for each $T > 0$. Then for all $\mu,\nu \in \mathcal{M}_{lf}^-$ the limit
 \[
  \lim_{\e \to 0}  \frac{\psi(\cdot, \mu + \e \nu) - \psi(\cdot,\mu)}{\e} = D_{\nu}\psi(\cdot,\mu)
 \]
 exists in $L^p([0,T]; \C^m)$ for each $T > 0$. This limit satisfies
 \begin{align}\label{eq: mu derivative}
  D_{\nu}\psi(t,\mu) = \int_{[0,t)}K(t-s)\nu(ds) + \int_0^t K(t-s)(DR)(\psi(s,\mu))D_{\nu}\psi(s,\mu)ds,
 \end{align}
 where $DR(x) = \beta^{\top} + \frac{1}{2}\mathrm{diag}(\sigma_1^2x_1,\dots, \sigma_m^2x_m)$.
\end{Theorem}
\begin{proof}
 Note that $G(s) = (DR)(\psi(s,\mu))$ satisfies $G \in L_{loc}^2(\R_+; \C^{m \times m})$. Letting $p(t,x) = x$ and noting that
 $K \ast \nu \in L_{loc}^2(\R_+; \C^{m \times m})$, we may apply \cite[Theorem B.3]{MR4019885} to find a unique solution
 $f \in L_{loc}^2(\R_+; \C^{m})$ of
 $f = K \ast \nu + K \ast (G p(\cdot, f))$.
 By definition, this solution is precisely the unique solution of \eqref{eq: mu derivative}, i.e., $f = D_{\nu}\psi(\cdot, \mu)$.

 It suffices to show that
 $\lim_{\e \to 0}\left\| \Delta_{\e}(\cdot) - D_{\nu}\psi(\cdot, \mu) \right\|_{L^p([0,T])} = 0$,
 where $\Delta_{\e}(t) := \e^{-1}(\psi(t, \mu + \e \nu) - \psi(t,\mu))$. To prove this, first note that
 \begin{align}
    \notag &\ \Delta_{\e}(t) = (K \ast \nu)(t)
    \\ &\qquad + \int_0^t  K(t-s) \left( \int_0^1 (DR)(\tau\psi(s,\mu+\e\nu) + (1-\tau) \psi(s,\mu)) d\tau \right) \Delta_{\e}(s) ds. \label{eq: derivative mu}
 \end{align}
 Hence, by using $\|DR(x)\|_2 \leq C(1 + |x|)$ and $\|DR(u) - DR(v)\|_2 \leq C|u-v|$, we obtain
 \begin{align*}
     &\ \left| \Delta_{\e}(t) - D_{\nu}\psi(t,\mu)\right|
     \\ &\leq \int_0^t \|K(t-s)\|_2 \int_0^1 \| (DR)(\tau\psi(s,\mu+\e\nu) + (1-\tau) \psi(s,\mu))
     \\ &\hskip50mm - (DR)(\psi(s,\mu))\|_2 d\tau |D_{\nu}\psi(s,\mu)| ds
     \\ &\ \ \ + \int_0^t \| K(t-s)\|_2 \int_0^1 \|(DR)(\tau\psi(s,\mu+\e\nu)
     \\ &\hskip50mm + (1-\tau) \psi(s,\mu)) \|_2 d\tau | \Delta_{\e}(s) - D_{\nu}\psi(s,\mu)| ds
     \\ &\leq C \int_0^t \|K(t-s)\|_2 \left| \psi(s, \mu + \e \nu) - \psi(s,\mu) \right| |D_{\nu}\psi(s,\mu)|ds
     \\ &\ \ \ + C \int_0^t \|K(t-s)\|_2 (1 + |\psi(s,\mu + \e \nu)| + |\psi(s,\mu)|) \left| \Delta_{\e}(s) - D_{\nu}\psi(s,\mu)\right|ds.
    \end{align*}
    After a short computation we obtain
    \begin{align*}
        &\ |\Delta_{\e}(t) - D_{\nu}\psi(t,\mu)|^2
        \\ &\leq \| \psi(\cdot, \mu + \e \nu) - \psi(\cdot, \mu)\|_{L^2([0,T])}^2f(t) + \int_0^t k(t-s) | \Delta_{\e}(s) - D_{\nu}\psi(s,\mu)|^2 ds
    \end{align*}
    with
    \[
     k(t) = 8C\left( 1+ \| \psi(\cdot,\mu)\|_{L^2([0,T])}^2 + \sup_{\e \in (0,1)}\| \psi(\cdot,\mu + \e \nu)\|_{L^2([0,T])}^2 \right)\|K(t)\|_2^2
    \]
    and $f(t) = 2C \int_0^t \|K(t-s)\|_2^2 |D_{\nu}\psi(s,\mu)|^2 ds$.
    Let $\widetilde{R} \in L_{loc}^1(\R_+; \R)$ be the resolvent of the second kind for $k$. Arguing as in the proof of Theorem \ref{Theorem: continuity in mu} gives
    \begin{align*}
        |\Delta_{\e}(t) - D_{\nu}\psi(t,\mu)|^2
        &\leq \| \psi(\cdot, \mu + \e \nu) - \psi(\cdot, \mu)\|_{L^2([0,T])}^2f(t)
        \\ &\ + \| \psi(\cdot, \mu + \e \nu) - \psi(\cdot, \mu)\|_{L^2([0,T])}^2\int_0^t |\widetilde{R}(t-s)|f(s)ds.
    \end{align*}
    Hence we obtain
    \begin{align*}
       &\ \| \Delta_{\e} - D_{\nu}\psi(\cdot,\mu)\|_{L^p([0,T])}
      \\ &\ \leq C \| \psi(\cdot, \mu + \e \nu) - \psi(\cdot, \mu)\|_{L^2([0,T])} \left( \| f\|_{L^{p/2}([0,T])}^{1/2} + \| |\widetilde{R}| \ast f \|_{L^{p/2}([0,T])}^{1/2}\right)
      \\ &\ \leq C \| \psi(\cdot, \mu + \e \nu) - \psi(\cdot, \mu)\|_{L^2([0,T])}
      \\&\qquad \qquad \cdot \left( \| K\|_{L^p([0,T])} \| D_{\nu}\psi(\cdot, \mu)\|_{L^2([0,T])} + \| \widetilde{R}\|_{L^1([0,T])}^{1/2} \| f\|_{L^{p/2}([0,T])}^{1/2}\right)
    \end{align*}
    The assertion now follows from Theorem \ref{Theorem: continuity in mu}.
\end{proof}

\begin{Remark}
 By inspection of the proof one can see that the above results hold true for $\mu = u \delta_0$ with $u \in \C_-^m$ without the additional regularity assumption $[K]_{\eta,p,T} < \infty$.
\end{Remark}

\begin{Corollary}\label{corr: first moment}
 Let $(b,\beta, \sigma, K)$ be admissible parameters. Let $X$ be the Volterra square-root process with admissible parameters $(b,\beta, \sigma, K)$ and initial state $x_0\in \R_+^m$. Then
 \[
  \E[ X_t ] = \left( I_m + \int_0^t (E_{\beta^{\top}}(s))^{\top}\beta ds \right)x_0 + \left( \int_0^t (E_{\beta^{\top}}(s))^{\top}ds \right) b.
 \]
\end{Corollary}
\begin{proof}
 Taking $\mu = u \delta_0$ with $u\in \C_-^m$ and then noting $\psi(s,0) = 0$, we find for $D_{u}\psi(t,0) := D_{\mu}\psi(t,0)$
 \begin{align*}
  D_{u}\psi(t,0) &= \int_{[0,t)}K(t-s)\mu(ds) + \int_0^t K(t-s)(DR)(\psi(s,0))D_{u}\psi(s,0)ds
  \\ &= K(t)u + \int_0^t K(t-s)\beta^{\top} D_{u}\psi(s,0)ds,
 \end{align*}
 where we have used $DR(0) = \beta^{\top}$. Solving this linear Volterra equation gives $D_{\mu}\psi(t,0) = (E_{\beta^{\top}} \ast \mu) (t) = E_{\beta^{\top}}(t)u$. Hence we obtain
 \begin{align*}
     \E[ \langle u, X_t \rangle ]
     &= \frac{d}{d\e}|_{\e = 0} \E\left[ e^{ \langle X_t, \e u \rangle} \right]
\\ &= \langle u, x_0 \rangle + \int_0^t \langle x_0, (DR)(0) D_u\psi(s,0) \rangle ds + \int_0^t \langle b, D_u\psi(s,0) \rangle ds
\\ &= \langle u, x_0 \rangle + \int_0^t \langle (E_{\beta^{\top}}(s))^{\top}\beta x_0, u \rangle ds + \int_0^t \langle (E_{\beta^{\top}}(s))^{\top}b, u \rangle ds.
 \end{align*}
 Since $u$ is arbitrary, the assertion is proved.

\end{proof}

\section{Uniform moment and H\"older bounds}

\subsection{Uniform moment bounds}
Let $X$ be the Volterra square-root process with admissible parameters $(b,\beta, \sigma, K)$ and initial state $x_0$. In this section we prove uniform in time moment bounds on the process $X$ which extend \cite[Lemma 3.1]{ACLP19} where similar bounds have been obtained on finite time intervals $[0,T]$.

First observe that after taking expectations in \eqref{eq: VCIR} we arrive at a convolution equation for $\E[X_t]$ which has the unique solution
\begin{align}\label{eq: first moment}
 \E[X_t] = \left(I_m - \int_0^t R_{\beta}(s) ds\right)x_0 + \left(\int_0^t E_{\beta}(s) ds\right)b,
\end{align}
where $R_{\beta}, E_{\beta}$ are respectively defined by \eqref{eq: def Rbeta} and \eqref{eq: def Ebeta} with $B=\beta$, compare with \cite[Lemma 4.1]{ACLP19}.
\begin{Lemma}
  Let $(b,\beta, \sigma, K)$ be admissible parameters, and let $X$ be the Volterra square-root process with admissible parameters $(b,\beta, \sigma, K)$. Then $R_{\beta} = (E_{\beta^{\top}})^{\top}(-\beta)$ and $E_{\beta} = (E_{\beta^{\top}})^{\top}$.
\end{Lemma}
\begin{proof}
 Using Corollary \ref{corr: first moment} and \eqref{eq: first moment} gives
 \begin{align*}
  &\ \left(I_m - \int_0^t R_{\beta}(s) ds\right)x_0 + \left(\int_0^t E_{\beta}(s) ds\right)b
  \\ &= \left( I_m + \int_0^t (E_{\beta^{\top}}(s))^{\top}\beta ds \right)x_0 + \left( \int_0^t (E_{\beta^{\top}}(s))^{\top}ds \right) b.
 \end{align*}
 Letting $x_0 = 0$ gives for each $b \in \R_+^m$
 \[
  \left(\int_0^t E_{\beta}(s) ds\right)b
  = \left(\int_0^t (E_{\beta^{\top}}(s))^{\top}ds \right) b.
 \]
 Since $b$ is arbitrary and the cone $\R_+^m$ is generating (that is $\R^m = \R_+^m - \R_+^m$), we conclude $\int_0^t E_{\beta}(s) ds
  = \int_0^t (E_{\beta^{\top}}(s))^{\top}ds$.
 Taking now the derivative in $t$ and noting that the integrands are continuous on $(0,\infty)$ because $K$ is continuous, yields
 $E_{\beta}(t) = (E_{\beta^{\top}}(t))^{\top}$ for all $t > 0$. This proves the second identity. The first identity follows from
 $(E_{\beta^{\top}}(t))^{\top}(-\beta) = E_{\beta}(t)(-\beta) = R_{\beta}(t)$. This proves the assertion.
\end{proof}
It seems natural, that the above relations may also be  derived directly from the definition of $R_{\beta}, E_{\beta}$. However, we have not succeeded in this way. The above relations provide the following observation used throughout this section.
\begin{Remark}
Let $p\in [1,\infty]$. One has $E_{\beta} \in L^p(\R_+; \R^{m \times m})$ if and only if $E_{\beta^{\top}} \in L^p(\R_+; \R^{m \times m})$.
\end{Remark}
The next lemma shows that integrability of $E_{\beta}$ is sufficient for the boundedness of the first moment.
\begin{Lemma}\label{lemma: first moment}
 Let $X$ be the Volterra square-root process with admissible parameters $(b,\beta, \sigma, K)$ and initial state $x_0$. Then
\begin{align}
\left(I_{m}-\int_{0}^{t}R_{\beta}(s)ds\right)v\in\R_{+}^{m}\quad\mbox{and}\quad\left(\int_{0}^{t}E_{\beta}(s)ds\right)v\in\R_{+}^{m}\label{eq: cone invariance}
\end{align}
 holds for any $v\in\R_{+}^{m}$. Moreover, if $E_{\beta} \in L^1(\R_+;\R^{m \times m})$, then
 \begin{align*}
  \lim_{t \to \infty} \E[ X_t ] = \left( I_m - \int_0^{\infty} R_{\beta}(s)ds \right) x_0 + \left( \int_0^{\infty} E_{\beta}(s)ds \right) b
 \end{align*}
 and
 \[
  \sup_{t \geq 0}\E[ |X_t| ] \leq C_{\beta}\max\{|x_0|, |b|\},
 \]
 where
 \begin{equation}\label{eq: defi Cbeta}
 C_{\beta} := (1 + \|R_{\beta}\|_{L^1})\sqrt{m} + \|E_{\beta}\|_{L^1}.
 \end{equation}
\end{Lemma}
\begin{proof}
 Note that \eqref{eq: first moment} holds for all $x_0,b \in \R_+^m$. Taking $b = 0$ shows that $I_m - \int_0^t R_{\beta}(s)ds$ leaves $\R_+^m$ invariant. Taking $x_0 = 0$ shows that $\int_0^t E_{\beta}(s)ds$ leaves $\R_+^m$ invariant. If $E_{\beta}$ is integrable, then using $E_{\beta}(-\beta) = R_{\beta}$, we find that also $R_{\beta}$ is integrable and hence we can pass to the limit $t \to \infty$ in \eqref{eq: first moment}. This proves the desired convergence of the first moment. The last assertion then follows from
  \begin{align*}
     \E[ |X_t| ] \leq  \E\left[ \sum_{i=1}^m X_{i,t} \right]
      \leq \sqrt{m} | \E[ X_t]|
     \leq \sqrt{m} (1 + \| R_{\beta}\|_{L^1})|x_0| + \| E_{\beta} \|_{L^1} |b|.
 \end{align*}
\end{proof}
\begin{Remark}
 If $m = 1$ and $\beta < 0$, then using $E_{\beta} = R_{\beta}(-\beta)^{-1}$ combined with \eqref{eq: cone invariance} implies that $0 \leq \int_0^t R_{\beta}(s)ds \leq 1$ and hence $0 \leq \E[ X_{t} ] \leq x_0 + \frac{b}{|\beta|}$ for $t \geq 0$.
\end{Remark}

To prove the uniform boundedness of higher order moments, we use \cite[Lemma 2.5]{MR4019885} to see that \eqref{eq: VCIR} is equivalent to
\begin{align}\label{eq: VCIR 1}
 X_{t} = \left(I_m - \int_0^t R_{\beta}(s) ds\right)x_0 + \left(\int_0^t E_{\beta}(s) ds\right)b + \int_0^t E_{\beta}(t-s) \sigma(X_s)dB_{s}.
\end{align}
Based on this equivalent representation for the Volterra square-root process, we obtain the following simple observation.

\begin{Lemma}\label{lemma: bounded moments}
 Let $X$ be the Volterra square-root process with admissible parameters $(b,\beta, \sigma, K)$ and initial state $x_0 \in \R_+^m$. If $E_{\beta} \in L^1(\R_+; \R^{m \times m}) \cap L^2(\R_+; \R^{m\times m})$, then
 \begin{align*}
  \sup_{t \geq 0}\E\left[|X_t|^p\right] \leq C < \infty, \qquad \forall p \geq 2,
 \end{align*}
  where $C = C(p, m,x_0, b,\|R_{\beta}\|_{L^1(\R_+)}, \|E_{\beta}\|_{L^1(\R_+)}, \|E_{\beta}\|_{L^2(\R_+)}, \sigma)$ is some constant.
\end{Lemma}
\begin{proof}
 Note that
 \begin{align*}
     \E[|X_t|^p] &\leq 2^{p-1}\left(1 + \int_0^{\infty}\|R_{\beta}(s)\|_2 ds\right)^p|x_0|^p
     + 2^{p-1}\left(\int_0^{\infty} \|E_{\beta}(s)\|_2ds\right)^p |b|^p
     \\ &\ \ \ + c  2^{p-1}\E\left[\left( \int_0^t \| E_{\beta}(t-s)\|_{HS}^2 \| \sigma(X_s)\|_{HS}^2 ds \right)^{p/2} \right],
 \end{align*}
 where $c>0$ is a constant only depending on $m,p$, and the first two terms are finite since $E_{\beta}$ (and hence $R_{\beta})$ is integrable over $\R_+$. For the last term we obtain with $\sigma^* = \max\{\sigma_1^2, \dots, \sigma_m^2\}$
 \begin{align*}
     &\ \E\left[\left( \int_0^t \| E_{\beta}(t-s)\|_{HS}^2 \| \sigma(X_s)\|_{HS}^2 ds \right)^{p/2} \right]
     \\ &\leq (\sqrt{m}\sigma^*)^{p/2} \E\left[\left( \int_0^t \| E_{\beta}(t-s)\|_{HS}^2 |X_s| ds \right)^{p/2} \right]
     \\ &\leq (\sqrt{m}\sigma^*)^{p/2} \left( \int_0^t \|E_{\beta}(t-s)\|_{HS}^2 ds \right)^{p/2 - 1} \int_0^t \| E_{\beta}(t-s)\|_{HS}^2 \E[|X_s|^{p/2}] ds
     \\ &\leq (\sqrt{m}\sigma^*)^{p/2} \sup_{t \geq 0}\E[|X_t|^{p/2}] \left(  \int_0^{\infty}\|E_{\beta}(s)\|_{HS}^2 ds \right)^{p/2},
 \end{align*}
 where we have used $\|\sigma(x)\|_{HS}^2=\sum_{k=1}^m \sigma_k^2 x_k \le \sigma^* \sqrt{m}|x|$ and Jensen's inequality. Thus we have shown that
 \[
 1+\sup_{t \geq 0}\E[|X_t|^p] \leq C(p) (1+\sup_{t \geq 0}\E[|X_t|^{p/2}]),
 \]
 where $C(p)>0$ is a constant.
 Letting $n \in \N$ be the smallest integer with $p/2^{n} \leq 2$ gives
 \begin{align*}
     1+\sup_{t \geq 0}\E[|X_t|^p] &\leq \prod_{k=0}^{n-1}C\left(\frac{p}{2^k}\right) (1+\sup_{t \geq 0}\E[|X_t|^{p/2^n}])
     \\ &\leq \prod_{k=0}^{n-1}C\left(\frac{p}{2^k}\right) \left(1+\sup_{t \geq 0}\left(\E[|X_t|^{2}]\right)^{\frac{p}{2^{n+1}}}\right)
     \\ &\leq \prod_{k=0}^{n-1}C\left(\frac{p}{2^k}\right) \left(1+\sup_{t \geq 0}\left(1 + \E[|X_t|^{2}]\right)\right).
 \end{align*}
Thus, we can use the previous estimate to find
 $\sup_{t \geq 0}\E[|X_t|^2] \leq C(2)(1+\sup_{t \geq 0}\E[|X_t|])$. This proves the assertion.
\end{proof}

\subsection{Uniform bound on H\"older increments}

It follows from \cite[Lemma 2.4]{MR4019885} that the Volterra square-root process
has H\"older continuous sample paths.
Below we recall their key estimate.
\begin{Proposition}\label{prop: Hoelder local}
 Let $X$ be the Volterra square-root process with admissible parameters $(b,\beta, \sigma, K)$  and initial state $x_0 \in \R_+^m$. Then for each $T > 0$ and $p \geq 2$ there exists a constant $C(T,p) > 0$ such that for all $s,t \in [0,T]$ with $0 \leq t-s \leq 1$ one has
 $\E[ |X_t - X_s|^p ] \leq C(T,p)(t-s)^{\gamma p/2}$.
\end{Proposition}
To prove the existence of a stationary process and absolute continuity of the limiting distribution, we need a similar estimate but with a constant $C(T)$ independent of $T$. For this purpose we need a stronger assumption then condition (v) in the definition of admissible parameters. Namely, we impose the following condition on $K$:
\begin{enumerate}
    \item[(K)] There exists a constant $C_3 > 0$ such that
    \[
     \int_0^{\infty} | K_i(r+h) - K_i(r)|^2 dr \leq C_3 h^{\gamma}
    \]
    holds for all $i = 1,\dots, m$, $h \in (0,1]$.
\end{enumerate}
Note that this condition is satisfied for $K_i(t) = t^{H-1/2}/\Gamma(H+1/2)e^{-\lambda t}$ with $\lambda \geq 0$, $H \in (0,1/2)$, where $\gamma = 2H$. Also, under additional conditions on $\beta$ we may  obtain $E_{\beta} \in L^1(\R_+; \R^{m \times m}) \cap L^2(\R_+; \R^{m \times m})$ as required in the subsequent proposition (compare with Section 5.3).

The next proposition gives the desired global H\"older estimate for the process.
\begin{Proposition}\label{prop: Hoelder global}
 Let $X$ be the Volterra square-root process with admissible parameters $(b,\beta, \sigma,K)$  and initial state $x_0 \in \R_+^m$. Suppose that condition (K) holds, and $E_{\beta} \in L^1(\R_+; \R^{m\times m}) \cap L^2(\R_+; \R^{m\times m})$. Then there exists for each $p \geq 2$ a constant $C(p) > 0$ such that
 $\E\left[ |X_t - X_s|^p \right] \leq C(p)(t-s)^{ \gamma p /2 }$ holds for all $s,t \geq 0$ with $0 \leq t-s \leq 1$.
\end{Proposition}
\begin{proof}
 Using \eqref{eq: VCIR 1} we find for $t,s \geq 0$ with $0 \leq t-s \leq 1$
 that $X_t - X_s = I_1 + I_2 + I_3 + I_4$ with
 $I_1 = -\int_s^t R_{\beta}(r)x_0 dr$, $I_2 = \int_s^t E_{\beta}(r)b dr$,
 and
 \begin{align*}
  I_3 &= \int_0^s ( E_{\beta}(t-r) - E_{\beta}(s-r))\sigma(X_r)dB_r
  \\ I_4 &= \int_s^t E_{\beta}(t-r)\sigma(X_r)dB_r.
 \end{align*}
 For the first term we obtain from Lemma \ref{lemma: global Ebeta}
 \begin{align*}
     |I_1|^p &\leq |x_0|^p \left( \int_s^t \| R_{\beta}(r)\|_2 dr \right)^p
     \\ &\leq |x_0|^p (t-s)^{p/2} \left(\int_s^t \| R_{\beta}(r)\|_2^2 dr \right)^{p/2}
     \\ &\leq |x_0|^p C^{p/2} (t-s)^{p/2 + \gamma p/2}.
 \end{align*}
 Likewise we obtain for the second term
 \begin{align*}
     |I_2|^p &\leq |b|^p (t-s)^{p/2} \left( \int_s^t \|E_{\beta}(r)\|_2^2 dr \right)^{p/2}
     \leq |b|^p C^{p/2} (t-s)^{p/2 + \gamma p/2}.
 \end{align*}
 For the third term we use the Burkholder-Davis-Gundy inequality to find
 \begin{align*}
     \E[|I_3|^p] &\leq c(m,p) \E\left[ \left( \int_0^s \|E_{\beta}(t-r) - E_{\beta}(s-r)\|_{HS}^2 \| \sigma(X_r)\|_{HS}^2 dr \right)^{p/2} \right]
     \\ &\leq (\sqrt{m}\sigma^*)^{p/2}c(m,p) \E\left[ \left( \int_0^s \|E_{\beta}(t-r) - E_{\beta}(s-r)\|_{HS}^2 |X_r| dr \right)^{p/2} \right]
     \\ &\leq (\sqrt{m}\sigma^*)^{p/2}c(m,p) \left( \int_0^s \|E_{\beta}(t-r) - E_{\beta}(s-r)\|_{HS}^2 dr \right)^{p/2-1}
     \\ & \qquad \qquad \qquad \cdot \int_0^s \| E_{\beta}(t-r) - E_{\beta}(s-r)\|_{HS}^2 \E\left[|X_r|^{p/2}\right]dr
     \\ &\leq (\sqrt{m}\sigma^*)^{p/2}c(m,p) \sup_{t \geq 0} \E\left[|X_r|^{p/2}\right] C^{p/2} (t-s)^{\gamma p /2},
 \end{align*}
 where we have used $\|\sigma(x)\|_{HS}^2
 \leq \sigma^* \sqrt{m} |x|$ with $\sigma^* = \max\{\sigma_1^2, \dots, \sigma_m^2\}$. Similarly we obtain for the last term
 \begin{align*}
     \E\left[ |I_4|^p \right] &\leq (\sqrt{m}\sigma^*)^{p/2}c(m,p) \sup_{t \geq 0} \E\left[|X_r|^{p/2}\right] \left( \int_s^t \|E_{\beta}(t-r)\|_{HS}^2 dr \right)^{p/2}
     \\ &\leq (\sqrt{m}\sigma^*)^{p/2}c(m,p) \sup_{t \geq 0} \E\left[|X_r|^{p/2}\right] C^{p/2} (t-s)^{\gamma p /2}.
 \end{align*}
 Combining all estimates and invoking Lemma \ref{lemma: bounded moments} proves the assertion.
\end{proof}

\section{Limiting distributions}

\subsection{Existence of limiting distributions}

Below under an integrability condition on $E_\beta$, we prove the existence of limiting distributions for the Volterra square-root process and therefore provide a mathematical justification of the mean-reversion property.
As a first step we prove the convergence of the Fourier-Laplace transform.
\begin{Proposition}\label{prop: convergence FT}
 Let $X$ be the Volterra square-root process with admissible parameters $(b,\beta,\sigma, K)$ and initial state $x_{0}\in\R_{+}^{m}$.
 Suppose that there exists $\eta \in (0,1)$ such that
 $[K]_{\eta,2,T} < \infty$ for each $T > 0$.
 Then the following assertions hold:
 \begin{enumerate}
     \item[(a)] Suppose that $E_{\beta} \in L^1(\R_+; \R^{m \times m})$. Let $\mu \in \mathcal{M}_{lf}^-$ be such that $|\mu|(\R_+) < \infty$ and $\mathrm{Im}(\mu) = 0$. Then $\psi \in L^1(\R_+; \C^m) \cap L^2(\R_+; \C^m)$, and
     \begin{align} \notag
      &\ \lim_{t \to \infty}
      \E\left[ e^{ \int_{[0,t]}\langle X_{t-s}, \mu(ds) \rangle} \right]
      \\ &= \exp\left\{ \langle x_0, \mu(\R_+) \rangle + \int_0^{\infty} \langle x_0, R(\psi(s,\mu)) \rangle ds + \int_0^{\infty} \langle b, \psi(s,\mu) \rangle ds \right\} \label{eq: FT convergence}
      \\ &= \exp\left\{ \langle A(\beta,x_0,b), \mu(\R_+) \rangle + \sum_{i=1}^m \frac{\sigma_i^2}{2}A_i(\beta,x_0,b)\int_0^{\infty} \psi_i(s,\mu)^2 ds \right\}, \label{eq: FT convergence II}
     \end{align}
     where $\psi(\cdot, \mu)$ denotes the unique solution of \eqref{eq: riccati extension},
     \begin{align}\label{eq: Abeta}
      A(\beta,x_0,b) = \left( I_m - \int_0^{\infty}R_{\beta}(s)ds \right)x_0 + \left( \int_0^{\infty}E_{\beta}(s)ds \right)b,
     \end{align}
     and $A_i(\beta,x_0,b)$ denotes the $i$-th component of the vector $A(\beta,x_0,b)$.

    \item[(b)] Suppose that $E_{\beta} \in L^1(\R_+; \R^{m \times m}) \cap L^2(\R_+; \R^{m \times m})$.
    Let $\mu \in \mathcal{M}_{lf}^-$ with $|\mu|(\R_+) < \infty$. Then $\psi \in L^1(\R_+; \C_-^m) \cap L^2(\R_+; \C_-^m)$, and the identities \eqref{eq: FT convergence} and \eqref{eq: FT convergence II} still hold.
 \end{enumerate}
\end{Proposition}
\begin{proof}
  (a) Instead of Fourier-Laplace transform, let us first reformulate the affine formula for the Laplace transform. The latter formulation is more natural to exploit the nonnegativity of the process. Noting $-\mu$ is an $\R_+^m$-valued measure, we let $V(t,\mu)= -\psi(t,\mu) \in \R_+^m$ and $\widetilde{R}_i(x) := -R_i(-x) = \langle x, \beta^i \rangle - \frac{\sigma_i^2}{2}x_i^2$, where $i = 1,\dots, m$ and $x \in \R_+^m$. Then $V(t,\mu)$ satisfies the Riccati-Volterra equation
  \begin{align*}
    V_i(t,\mu) = -\int_{[0,t)} K_i(t-s)\mu_i(ds) + \int_0^t K_i(t-s) \widetilde{R}_i(V(s,\mu))ds,
  \end{align*}
  and, by Corollary \ref{Theorem: affine riccati extension}, it holds
  \begin{align}
   & \mathbb{E}\left[\mathrm{e}^{- \int_{[0,t]} \langle X_{t-s}, -\mu(ds)\rangle}\right] \notag \\
   &\quad = \exp\left\{ - \langle x_0,- \mu([0,t]) \rangle - \int_0^t \langle x_0,\widetilde{R}(V(s,\mu)) \rangle ds -\int_{0}^{t}\langle b, V(s,\mu) \rangle ds\right\},\label{eq: laplace for volterra squre root}
  \end{align}
  where $\widetilde{R} = (\widetilde{R}_1, \dots, \widetilde{R}_m)^{\top}$.
  Applying Jensen's inequality, we have
  \[
   e^{-\int_{[0,t]} \langle \E[X_{t-s}], -\mu(ds)\rangle } \leq \mathbb{E}\left[\mathrm{e}^{- \int_{[0,t]} \langle X_{t-s}, -\mu(ds)\rangle}\right],
  \]
  and hence
  \begin{align*}
    &\ \langle x_0, -\mu([0,t]) \rangle + \int_0^t \langle x_0,\widetilde{R}(V(s,\mu)) \rangle ds + \int_{0}^{t}\langle b, V(s,\mu) \rangle ds
    \\ &\leq \int_{[0,t]} \langle \E[X_{t-s}], -\mu(ds)\rangle
    \\ &\leq |\mu|(\R_+) \sup_{t\geq 0}\E[ |X_t| ]
    \\ &\leq |\mu|(\R_+) C_{\beta}\max\{|x_0|,|b|\}.
  \end{align*}
  Note that this inequality holds for all choices of $b, x_0 \in \R_+^m$. In particular, choosing
  $b = (1,\dots, 1)^{\top}$ and $x_0 = 0$ gives
  \begin{align}\label{eq: V estimate}
   \sum_{i=1}^m \int_0^t V_i(s,\mu) ds
   \leq \sqrt{m}|\mu|(\R_+) C_{\beta}.
  \end{align}
  To estimate the integral involving $\widetilde{R}$, let us first note that the left-hand side of \eqref{eq: laplace for volterra squre root} is bounded by $1$ which gives
  \[
   \langle x_0, -\mu([0,t]) \rangle + \int_0^t \langle x_0,\widetilde{R}(V(s,\mu)) \rangle ds + \int_{0}^{t}\langle b, V(s,\mu) \rangle ds \geq 0
  \]
  for all $x_0,b \in \R_+^m$. For $b = 0$ and $x_0 = (1,\dots, 1)^{\top}$ we obtain
\begin{align*}
 \sum_{i=1}^m \frac{\sigma_i^2}{2}\int_0^t V_i(s,\mu)^2 ds
 &\leq \sum_{i=1}^m \left(-\mu_i([0,t])\right) + \sum_{i=1}^m \int_0^t \langle V(s,\mu), \beta^i \rangle ds
 \\ &\leq \sqrt{m}|\mu|(\R_+) + \sum_{i = 1}^m |\beta^i| \int_0^t |V(s,\mu)| ds
 \\ &\leq \sqrt{m}|\mu|(\R_+)  + m \| \beta\|_{HS} |\mu|(\R_+) C_{\beta},
\end{align*}
where we have used \eqref{eq: V estimate}. In view of \eqref{eq: V estimate} and the particular form of $\tilde{R}$, we obtain
\begin{align}
 \int_0^t |\widetilde{R}(V(s,\mu))| ds
 &\leq  \sum_{i=1}^m \int_0^t|\langle V(s,\mu), \beta^i \rangle |ds + \sum_{i=1}^m \frac{\sigma_i^2}{2} \int_0^t V_i(s,\mu)^2 ds \notag
 \\ &\leq \left(\sqrt{m} + 2m \| \beta\|_{HS}  C_{\beta} \right)|\mu|(\R_+).\label{eq: R estimate}
\end{align}
 This estimate combined with \eqref{eq: V estimate} proves the convergence in \eqref{eq: FT convergence}. For the second identity in part (a), use the second identity from Corollary \ref{Theorem: affine riccati extension} to pass to the limit $t \to \infty$, i.e., we show that
 \begin{enumerate}
     \item[(i)] $\lim_{t \to \infty} \int_{[0,t]}\langle\E[X_{t-s}], \mu(ds)\rangle = \langle A(\beta,x_0,b),\mu(\R_+)\rangle$;

     \item[(ii)] $\lim_{t \to \infty} \int_0^t \E[ X_{i,t-s}] V_i(s,\mu)^2 ds = A_i(\beta,x_0,b) \int_0^{\infty}V_i(s,\mu)^2 ds$.
 \end{enumerate}
 So let $\e > 0$. Since $A(\beta, x_0,b) = \lim_{t \to \infty}\E[X_t]$ by Lemma \ref{lemma: first moment}, we find $t_0 > 0$ such that $|\E[X_{i,r}] - A_i(\beta,x_0,b)| < \e$ for all $r \geq t_0$ and $i=1,\ldots,m$. Then for all $t \geq 2t_0$ we obtain
 \begin{align*}
     &\ \left| \int_{[0,t]}\langle \E[X_{t-s}], \mu(ds)\rangle - \langle A(\beta,x_0,b),\mu(\R_+)\rangle \right|
     \\ &\leq \left| \int_{[0,t/2]} \langle \E[X_{t-s}] - A(\beta,x_0,b), \mu(ds)\rangle \right|
     \\ &\ \ \ + \left| \int_{(t/2,t]} \langle \E[X_{t-s}] - A(\beta,x_0,b), \mu(ds)\rangle \right|
      +\left|\langle A(\beta,x_0,b),\mu((t,\infty))\rangle \right|
     \\ &\leq \e \sqrt{m} |\mu|(\R_+) + \left( \sup_{s \geq 0}\E[|X_s|]\right)|\mu|((t/2,t]) + 2 |A(\beta,x_0,b)|\,|\mu|((t/2,\infty)).
 \end{align*}
 Since $|\mu|(\R_+) < \infty$, we have $|\mu|((t/2,\infty)) \to 0$ as $t \to \infty$, which proves (i). For (ii) we have
 \begin{align*}
     &\ \left| \int_0^t \E[ X_{i,t-s}] V_i(s,\mu)^2 ds - A_i(\beta,x_0,b) \int_0^{\infty}V_i(s,\mu)^2 ds \right|
     \\ &\leq \int_0^t \left|\E[ X_{i,t-s}] - A_i(\beta,x_0,b)\right| V_i(s,\mu)^2 ds
     + A_i(\beta, x_0,b) \int_t^{\infty} V_i(s,\mu)^2 ds.
 \end{align*}
 The second term tends to zero due to $\int_0^{\infty}V_i(s,\mu)^2 ds < \infty$.
 For the first term we have for all $t \geq 2t_0$,
 \begin{align*}
     &\ \int_0^t \left|\E[ X_{i,t-s}] - A_i(\beta,x_0,b)\right| V_i(s,\mu)^2 ds
     \\ &= \int_0^{t/2} \left|\E[ X_{i,t-s}] - A_i(\beta,x_0,b)\right| V_i(s,\mu)^2 ds
     \\ & \ \ \ + \int_{t/2}^t \left|\E[ X_{i,t-s}] - A_i(\beta,x_0,b)\right| V_i(s,\mu)^2 ds
     \\ &\leq \e \int_0^{\infty} V_i(s,\mu)^2 ds + \left( \sup_{s \geq 0}\E[|X_s|] + |A(\beta,x_0,b)|\right) \int_{0}^{\infty} \1_{[t/2,t]}(s)V_i(s,\mu)^2 ds.
 \end{align*}
 The dominated convergence theorem implies that the second term tends to zero as $t \to \infty$. Since $\e$ is arbitrary, this proves (ii) and thus completes the proof of part (a).

 Part (b) can be shown in a similar way, since Theorem \ref{Theorem: riccati extension} applied for $T = \infty$ still provides the desired integrability $\int_0^{\infty}\left( |\psi(t,\mu)| + |\psi(t,\mu)|^2 \right)dt < \infty$.
\end{proof}

\begin{Remark}\label{coro: convergence laplace transform}
 If we choose  $\mu(ds) = u \delta_0(ds)$, then the statements of Proposition \ref{prop: convergence FT} and the estimates established in the above proof still hold even if we drop the condition $[K]_{\eta,2,T} < \infty$. Essentially, this condition was used to ensure the existence of $\psi(t,\mu)$ and the applicability of Theorem \ref{Theorem: riccati extension} and Corollary \ref{Theorem: affine riccati extension}. However, for the particular choice of $\mu(ds) = u \delta_0(ds)$, we can work directly with \eqref{eq: riccati} and $\psi(\cdot, u,0)$ instead of the extension \eqref{eq: riccati extension} and $\psi(\cdot, u \delta_0)$, and then apply Lemma \ref{lemma: local integrability psi}. The above proof still works in this case with some obvious adaptions.
\end{Remark}

From the convergence of the Fourier-Laplace transform we can now deduce convergence towards limiting distributions.
The following is our main result on limiting distributions for the Volterra square-root process. In contrast to the classical case, the limiting distribution now also involves the initial state of the process.
For this purpose we define
\begin{align}\label{defi:N}
    \mathcal{N} &= \left \{ v \in \R^m \ : \ \int_0^{\infty}R_{\beta}(t)v dt = v \right\},
\end{align}
and let $\mathcal{N}^{\perp}$ be the orthogonal complement of $\mathcal{N}$. Denote by $P$ the orthogonal projection operator onto $\mathcal{N}^{\perp}$, i.e., $\mathrm{ker}(P) = \mathcal{N}$ and $\mathrm{ran}(P) = \mathcal{N}^{\perp}$.
\begin{Theorem}\label{Theorem: limiting distribution}
 Let $X$ be the Volterra square-root process with admissible parameters $(b,\beta,\sigma, K)$ and initial state $x_{0}\in\R_{+}^{m}$. Suppose that $E_{\beta} \in L^1(\R_+; \R^{m \times m})$.
 Then the law of the random variable $X_t$ converges for $t \to \infty$ weakly to a limiting distribution $\pi_{x_0}$, whose Fourier-Laplace transform is for $u \in \C_-^m$ with $\mathrm{Im}(u) = 0$ given by
 \begin{align*}
  &\ \int_{\R_+^m} e^{\langle u, y\rangle} \pi_{x_0}(dy)
  \\ &\ = \exp\left\{ \langle x_0, u \rangle + \int_0^{\infty} \langle x_0, R(\psi(s,u \delta_0)) \rangle ds + \int_0^{\infty} \langle b, \psi(s,u\delta_{0}) \rangle ds \right\}.
 \end{align*}
 Moreover, $\pi_{x_0}$ has finite first moment and satisfies
 \begin{align}\label{eq: structure pi}
   \pi_{x_0} = \pi_{Px_0} = \pi_0 \ast \pi_{Px_0}^{b=0},
 \end{align}
 where $\ast$ denotes the usual convolution of probability measures on $\R_+^m$, and $\pi_{Px_0}^{b=0}$ is the limiting distribution of the Volterra square-root process with admissible parameters $(b = 0, \beta, \sigma, K)$ and initial state $Px_0$.
 Finally, if, in addition, $E_{\beta} \in L^2(\R_+; \R^{m\times m})$, then the Fourier-Laplace transform representation for $\pi_{x_0}$ can be extended to all $u \in \C_-^m$.
\end{Theorem}
\begin{proof}
 Consider $u \in \C_-^m$ with $\mathrm{Im}(u) = 0$. According to Remark \ref{coro: convergence laplace transform},  if we take $\mu(ds) = u \delta_0(ds)$, then it holds
 \begin{align*}
   &\ \lim_{t \to \infty}
   \E\left[ e^{ \langle X_{t}, u \rangle} \right]
   \\ &\ = \exp\left\{ \langle x_0, u \rangle + \int_0^{\infty} \langle x_0, R(\psi(s,u \delta_0)) \rangle ds + \int_0^{\infty} \langle b, \psi(s,u \delta_0) \rangle ds \right\}.
 \end{align*}
 Moreover, the estimates \eqref{eq: V estimate} and \eqref{eq: R estimate} hold with $|\mu|(\R_+) = |u|$,
 showing that the right-hand side is continuous at $u = 0$. Hence using L\'evy's continuity theorem for Laplace transforms proves that $X_t$ converges weakly to some distribution $\pi_{x_0}$ and that the desired formula for the Laplace transform of $\pi_{x_0}$ holds.
 The extension to the Fourier-Laplace transform with $u \in \C_-^m$ follows now from Proposition \ref{prop: convergence FT}.(b). An application of the Lemma of Fatou shows that the limit distribution $\pi_{x_0}$ has finite first moment, i.e., $\int_{\R_+^m}|x|\pi_{x_0}(dx) \leq \sup_{t \geq 0}\E[|X_t|] < \infty$. It remains to prove \eqref{eq: structure pi}.
 For this purpose we use the second identity from Proposition \ref{prop: convergence FT}, i.e.,
  \begin{align*}
  &\ \lim_{t \to \infty}
   \E\left[ e^{ \langle X_{t}, u \rangle} \right]
  \\ &\ = \exp\left\{ \langle A(\beta,x_0,b), u \rangle + \sum_{i=1}^m \frac{\sigma_i^2}{2}A_i(\beta,x_0,b)\int_0^{\infty} \psi_i(s,u \delta_0)^2 ds \right\},
 \end{align*} 
 where $A(\beta,x_0,b)$ is defined in \eqref{eq: Abeta}.
  Then noting that
  \[ 
   A(\beta, x_0, b) = A(\beta, Px_0, b) = A(\beta, 0, b) + A(\beta, Px_0, 0)
  \] 
  readily yields \eqref{eq: structure pi} due to uniqueness of the Laplace transform. This completes the proof.
\end{proof}

\subsection{Stationary process}

Next we construct the corresponding stationary process. \begin{Theorem}\label{Theorem: stationary process}
Let $X$ be the Volterra square-root process with admissible parameters
$(b,\beta,\sigma,K)$ and initial state $x_{0}\in\R_{+}^{m}$. Suppose
that $E_{\beta}\in L^{1}(\R_{+};\R^{m\times m})\cap L^{2}(\R_{+};\R^{m\times m})$,
condition (K) holds, and there exists $\eta\in(0,1)$ such that $[K]_{\eta,2,T}<\infty$
for each $T>0$. Then there exists a stationary process $X^{\mathrm{stat}}$ with continuous sample paths such that the following assertions hold:
\begin{enumerate}
\item[(a)] It holds that $(X_{t+h})_{t \geq 0} \Rightarrow (X_t^{\mathrm{stat}})_{t \geq 0}$ weakly on $C(\R_+; \R^m_+)$ as $h \to \infty$.
\item[(b)] The finite dimensional distributions of $X^{\mathrm{stat}}$ are determined by
\begin{align*}
 &\ \E\left[e^{\sum_{j=1}^{n}\langle X_{t_{j}}^{\mathrm{stat}},u_{j}\rangle}\right] &
\\ &\ \ = \exp\bigg\{ \sum_{j=1}^n \langle A(\beta, x_0,b), u_j \rangle
+ \sum_{i=1}^m \frac{\sigma_i^2}{2}A_i(\beta,x_0,b) \int_0^{\infty}\psi_i(s)^2ds \bigg \},
\end{align*}
where $A(\beta,x_0,b)$ is defined by \eqref{eq: Abeta},
$\psi(\cdot) = \psi(\cdot, \mu_{t_1,\dots, t_n})$ denotes the unique solution of \eqref{eq: riccati extension} with  $\mu_{t_{1},\dots,t_{n}}(ds)=\sum_{j=1}^{n}u_{j}\delta_{t_{n}-t_{j}}(ds)$, $n\in\N$, $u_{1},\dots,u_{n}\in\C_{-}^{m}$ and $0\leq t_{1}<\dots<t_{n}$.
\end{enumerate}
\end{Theorem}
\begin{proof}
Choose $p\geq2$ sufficiently large so that $\gamma p>2$. By Proposition
\ref{prop: Hoelder global}, we find some constant $C(p)>0$ such
that $\E[|X_{t}-X_{s}|^{p}]\leq C(p)(t-s)^{\gamma p/2}$ holds for all $t,s\geq0$ with $0\leq t-s\leq1$. Define for $h\geq0$ the process $X^{h}$ by $X_{t}^{h}=X_{h+t}$,  where $t \ge 0$.
Then $X^{h}$ has continuous sample paths and satisfies
$\sup_{h\geq0}\E[|X_{t}^{h}-X_{s}^{h}|^{p}] \le C(p)(t-s)^{\gamma p/2}$
for $0\leq t-s\leq1$. Applying the Kolmogorov tightness criterion
(see, e.g., \cite[Corollary 16.9]{MR1876169}) shows that $(X^{h})_{h\geq0}$
is tight on $C(\R_{+};\R_{+}^{m})$.

Hence we conclude that along a sequence $h_{k}\uparrow\infty$, $X^{h_{k}}$
converges in law to some continuous process $X^{\mathrm{stat}}$.
Take $n\in\N$ and let $0\leq t_{1}<\dots<t_{n}$. Applying Proposition
\ref{prop: convergence FT} for the particular choice $\mu_{t_{1},\dots,t_{n}}(ds)=\sum_{j=1}^{n}u_{j}\delta_{t_{n}-t_{j}}(ds)$,
where $u_{1},\dots,u_{m}\in\C_{-}^{m}$,
we find that, for all $h\ge0$,
\begin{align}
\E\left[e^{\sum_{j=1}^{n}\langle X_{t_{j}+h}^{\mathrm{stat}},u_{j}\rangle}\right] & =\lim_{k\to\infty}\E\left[e^{\sum_{j=1}^{n}\langle X_{t_{j}+h}^{h_{k}},u_{j}\rangle}\right]\nonumber \\
 & =\lim_{k\to\infty}\E\left[e^{\sum_{j=1}^{n}\langle X_{h_{k}+h+t_{j}},u_{j}\rangle}\right]\label{eq: weak conver of finite}\\
 & =\lim_{k\to\infty}\E\left[e^{\int_{[0,h_{k}+h+t_{n}]}\langle X_{h_{k}+h+t_{n}-s},\mu_{t_{1},\dots,t_{n}}(ds)\rangle}\right]\nonumber \\
 & =\exp\bigg\{\sum_{j=1}^{n}\langle x_{0},u_{j}\rangle+\int_{0}^{\infty}\langle x_{0},R(\psi(s,\mu_{t_{1},\dots,t_{n}}))\rangle ds\nonumber \\
 & \hskip35mm+\int_{0}^{\infty}\langle b,\psi(s,\mu_{t_{1},\dots,t_{n}})\rangle ds\bigg\}.\label{eq: weak conv of finite II}
\end{align}
In view of \eqref{eq: FT convergence II}, \eqref{eq: weak conver of finite} and
\eqref{eq: weak conv of finite II}, the desired formula of the Fourier transform is proved. Since $\{h_k\}$ is arbitrary and \eqref{eq: weak conv of finite II} is independent of $\{h_k\}$, it is  standard to verify the weak convergence in (a). The assertion is proved.
\end{proof}
A direct consequence of Theorem \ref{Theorem: stationary process} is that $X_t \longrightarrow \pi_{x_0}$ weakly as $t\to \infty$, and $X_t^{\mathrm{stat}}$ has distribution $\pi_{x_0}$ for each $t \geq 0$. In the next statement we compute the moments, covariance structure, and autocovariance function of the stationary process.
\begin{Corollary}
 Under the same conditions as in Theorem \ref{Theorem: stationary process},
 the stationary process $X^{\mathrm{stat}}$ satisfies $\E[ |X_t^{\mathrm{stat}}|^p] = \int_{\R_+^m}|x|^p \pi_{x_0}(dx) < \infty$ for each $p \geq 2$. Moreover, its first moment is given by
 \[
  \E[X_t^{\mathrm{stat}}] = \left( I_m - \int_0^{\infty}R_{\beta}(s)ds \right)x_0 + \left( \int_0^{\infty}E_{\beta}(s)ds \right)b,
 \]
 while its autocovariance function is, for $0 \leq s \leq t$, given by
 \begin{align*}
  &\ \mathrm{cov}(X_t^{\mathrm{stat}}, X_s^{\mathrm{stat}})
  \\ &= \int_0^{\infty} E_{\beta}(t-s+u)\sigma(A(\beta,x_0,b))\sigma(A(\beta,x_0,b))^{\top}E_{\beta}(u)^{\top}du.
 \end{align*}
\end{Corollary}
\begin{proof}
 Since $\sup_{t \geq 0}\E[|X_t|^p] < \infty$ holds for each $p \geq 2$ and $X_t \longrightarrow \pi_{x_0}$ weakly, the Lemma of Fatou implies that $\pi_{x_0}$ has all finite moments. Since $X^{\mathrm{stat}}$ is stationary, we conclude the first assertion. For the first moment formula we note that $\E[X_t] \longrightarrow A(\beta,x_0,b)$ as $t \to \infty$. Since $\sup_{t \geq 0}\E[|X_t|^2] < \infty$, we easily conclude that $\lim_{t \to \infty}\E[X_t] = \int_{\R_+^m}x \pi_{x_0}(dx) = \E[X_t^{\mathrm{stat}}]$. This proves the desired first moment formula for the stationary process. Noting that $X_t - \E[X_t] = \int_0^s E_{\beta}(t-u)\sigma(X_u)dB_u + \int_s^t E_{\beta}(t-u)\sigma(X_u)dB_u$ and
 $X_s - \E[X_s] = \int_0^s E_{\beta}(s-u)\sigma(X_u)dB_u$, we find that the autocovariance function for $X$ is given by
 \begin{align*}
     \mathrm{cov}(X_t,X_s) &= \E\left[ \left(\int_0^s E_{\beta}(t-u)\sigma(X_u)dB_u \right) \left( \int_0^s E_{\beta}(s-u)\sigma(X_u)dB_u \right)^{\top} \right]
     \\ &= \int_0^s E_{\beta}(t-u)\E\left[\sigma(X_u)\sigma(X_u)^{\top}\right]E_{\beta}(s-u)^{\top}du
     \\ &= \int_0^s E_{\beta}(t-u)\sigma(\E[X_u])\sigma(\E[X_u])^{\top}E_{\beta}(s-u)^{\top}du,
 \end{align*}
 where we have used the particular form of $\sigma(x)$ so that $\E\left[\sigma(X_u)\sigma(X_u)^{\top}\right]= \sigma(\E[X_u])\sigma(\E[X_u])^{\top}$.
 Thus the autocovariance function of the stationary process is given by
  \begin{align*}
   &\ \mathrm{cov}(X_t^{\mathrm{stat}}, X_s^{\mathrm{stat}})
   \\ &= \lim_{h \to \infty}\mathrm{cov}(X_{t+h},X_{s+h})
  \\ &= \lim_{h\to \infty}\int_0^{s+h} E_{\beta}(t+h-u)\sigma(\E[X_u])\sigma(\E[X_u])^{\top}E_{\beta}(s+h-u)^{\top}du
  \\ &= \lim_{h \to \infty}\int_{-h}^{s}E_{\beta}(t-u)\sigma(\E[X_{u+h}])\sigma(\E[X_{u+h}])^{\top}E_{\beta}(s-u)^{\top}du
  \\ &= \int_{-\infty}^s E_{\beta}(t-u)\sigma(A(\beta,x_0,b))\sigma(A(\beta,x_0,b))^{\top}E_{\beta}(s-u)^{\top}du
  \\ &=\int_0^{\infty} E_{\beta}(t-s+u))\sigma(A(\beta,x_0,b)\sigma(A(\beta,x_0,b))^{\top}E_{\beta}(u)^{\top}du,
 \end{align*}
 which proves the assertion.
\end{proof}

The particular form of the Laplace transform for the limiting distribution and the stationary process $X^{\mathrm{stat}}$ give the following characterization for the independence on the initial condition $x_0$.
\begin{Corollary}
 Let $(b,\beta, \sigma, K)$ be admissible parameters. Suppose that $E_{\beta}$ belongs to $L^1(\R_+; \R^{m \times m}) \cap L^2(\R_+; \R^{m \times m})$, then the following are equivalent:
 \begin{enumerate}
     \item[(i)] The stationary process $X^{\mathrm{stat}}$ is independent of $x_0$;
     \item[(ii)] The limiting distribution $\pi_{x_0}$ is independent of $x_0$;
     \item[(iii)] The function $x_0 \longmapsto \int_{\R_+^m}x \pi_{x_0}(dx)$ is constant;
     \item[(iv)] $\int_0^{\infty} R_{\beta}(t)dt = I_m$.
 \end{enumerate} 
\end{Corollary}
\begin{proof}
 Since $\pi_{x_0}$ is the law of $X_t^{\mathrm{stat}}$, clearly (i) implies (ii), and (ii) implies (iii). Suppose that (iii) holds. Using the first moment for the stationary process, we have $A(\beta,x_0,b) = \int_{\R_+^m} y \pi_{x_0}(dy) = \int_{\R_+^m}y \pi_{0}(dy) = A(\beta,0,b)$. The particular form  of $A(\beta,x_0,b)$ readily yields (iv). Finally, suppose that (iv) is satisfied. Then $A(\beta,x_0,b) = \int_0^{\infty} E_{\beta}(t)bdt$ is independent of $x_0$ and hence the Laplace transform for the stationary process implies that $X^{\mathrm{stat}}$ is independent of $x_0$, i.e., (i) holds.
\end{proof}
Finally, we discuss implications and also a sufficient condition for (iv).
\begin{Theorem}\label{Theorem: limiting distribution characterisation}
 Let $X$ be a Volterra square-root process with admissible parameters $(b,\beta, \sigma, K)$, and suppose that $E_{\beta} \in L^1(\R_+; \R^{m \times m})$. Then the following assertions hold:
 \begin{enumerate}
      \item[(a)] If $\int_0^{\infty}R_{\beta}(t)dt = I_m$, then $K\beta$ is not integrable on $\R_+$;
      \item[(b)] If $\beta^{\top} \beta$ has only strictly positive eigenvalues and the function $K_*(t) = \min_{j=1, \dots, m}K_j(t)$ is not integrable over $\R_+$, then $\int_0^{\infty}R_{\beta}(t)dt = I_m$.
 \end{enumerate}
\end{Theorem}
\begin{proof}
 (a) Assume $\int_0^{\infty}R_{\beta}(t)dt = I_m$.  Suppose $K_{\beta} := K(-\beta)$ is integrable on $\R_+$. The integrability of $R_{\beta}$ and $K_{\beta}$ on $\mathbb{R}_{+}$ implies that $\widehat{R}_{\beta}(z), \widehat{K}_{\beta}(z)$ are well defined for all $\mathrm{Re}(z) \geq 0$. Using the Paley-Wiener Theorem (see \cite[Chapter 2]{MR1050319}) shows that $\mathrm{det}(I_m + \widehat{K}_{\beta}(z)) \neq 0$. Solving \eqref{eq: def Rbeta} (for the Laplace transforms) yields 
 \begin{equation}\label{Laplace eq for RB and KB}
 \widehat{R}_{\beta}(z) = \widehat{K_{\beta}}(z)(I_m + \widehat{K_{\beta}}(z))^{-1}, \qquad \mathrm{Re}(z) \geq 0.
 \end{equation}
 Evaluating at $z = 0$ gives
 \[
  I_m = \int_0^{\infty} R_{\beta}(t)dt = \int_0^{\infty}K_{\beta}(t)dt \left( I_m + \int_0^{\infty}K_{\beta}(t)dt  \right)^{-1}.
 \]
 Hence $I_m + \int_{0}^{\infty}K_{\beta}(t)dt = \int_{0}^{\infty}K_{\beta}(t)dt$, which is impossible. So $K\beta$ is not integrable on $\R_+$.

 (b) Define $K^{\lambda}_{\beta}(t) = e^{-\lambda t}K_{\beta}(t)$ for $\lambda > 0$. Since $K_1,\dots, K_m$ are nonnegative and nonincreasing, it follows that $e^{-\lambda t}K_1(t), \dots, e^{-\lambda t}K_m(t)$ are integrable on $\R_+$. Hence also $K_{\beta}^{\lambda}$ is integrable on $\R_+$. Let
 $R_{\beta}^{\lambda}(t):= e^{-\lambda t}R_{\beta}(t)$. By the definition in \eqref{eq: def Rbeta}, it is easy to verify that $R_{\beta}^{\lambda}$ is the resolvent of the second kind of $K_{\beta}^{\lambda}$. As mentioned in Remark \ref{integrability of RB and EB}, we have $E_{\beta}(-\beta) = R_{\beta}$, so $R_{\beta} \in L^1(\R_+; \R^{m \times m})$. Similarly to \eqref{Laplace eq for RB and KB}, it holds
 \begin{align*}
  \int_0^{\infty} e^{-\lambda t}R_{\beta}(t)dt &= \int_0^{\infty}e^{-\lambda t}K_{\beta}(t)dt \left( I_m + \int_0^{\infty}e^{-\lambda t}K_{\beta}(t)dt  \right)^{-1}
  \\ &= I_m - \left( I_m + \int_0^{\infty} e^{- \lambda t}K_{\beta}(t)dt \right)^{-1}.
 \end{align*}
 Since $R_{\beta}$ is integrable, the left-hand side converges to $\int_0^{\infty}R_{\beta}(t)dt$ when $\lambda \searrow 0$. Thus it suffices to show that
 \begin{equation} \label{esti of inverse norm}
 \lim_{\lambda \searrow 0} \left( I_m + \int_0^{\infty} e^{- \lambda t}K_{\beta}(t)dt \right)^{-1} = 0.
 \end{equation}
 Note that
 \begin{align*}
  &\ \left\| \left( \int_0^{\infty} e^{- \lambda t}K_{\beta}(t)dt \right)^{-1} \right\|_{2}^{-2}
  \\ &= \inf_{|v| = 1} \left \langle  \left(\int_0^{\infty} e^{- \lambda t}K_{\beta}(t)dt \right) v, \left( \int_0^{\infty} e^{- \lambda t}K_{\beta}(t)dt \right)  v \right \rangle
  \\ &= \inf_{|v| = 1} \int_0^{\infty} \int_0^{\infty} e^{- \lambda (t+s)} \langle K(t)\beta v, K(s)\beta v \rangle ds dt
  \\ &\geq \inf_{|v| = 1} \int_0^{\infty} \int_0^{\infty} e^{-\lambda (t+s)} K_*(t)K_*(s) \langle v,  \beta^{\top} \beta v\rangle ds dt
  \\ &\geq \lambda_{\min}( \beta^{\top} \beta) \left( \int_0^{\infty} e^{- \lambda t} K_*(t) dt \right)^2,
 \end{align*}
 where $\lambda_{\min}( \beta^{\top} \beta) > 0$ denotes the smallest eigenvalue of $\beta^{\top} \beta$. We thus obtain
 $\left\| \left( \int_0^{\infty} e^{- \lambda t}K_{\beta}(t)dt \right)^{-1} \right\|_{2} \to 0$ as $ \lambda \searrow 0$. So
 \[
 \lim_{\lambda \searrow 0} \left\|  \int_0^{\infty} e^{- \lambda t}K_{\beta}(t)dt  \right\|_{2} \ge  \lim_{\lambda \searrow 0} \left\| \left( \int_0^{\infty} e^{- \lambda t}K_{\beta}(t)dt \right)^{-1} \right\|_{2}^{-1}=\infty,
 \]
 which easily implies \eqref{esti of inverse norm}. The assertion is proved.
\end{proof}

\subsection{Sufficient conditions and examples}

Next we provide some examples for admissible parameters $(b,\beta, \sigma, K)$ for which our results are applicable. Here we focus first on completely monotone and then integrable kernels.

By an abuse of notation, let $\langle \cdot, \cdot \rangle$ denote \textit{the inner product} on $\C^m$, i.e., $\langle v, w \rangle=\sum_{j=1}^m v_i\overline{w}_j $. For a matrix $A\in \C^{m\times m}$, we write $A\succeq 0$ if $\langle v,Av \rangle \ge 0$ for all $v\in \C^m$, and write $A\succ 0$ if $\langle v,Av \rangle > 0$ for all nonzero $v\in \C^m$. For another matrix $B\in \C^{m\times m}$ we write $A \succeq B $ if $A-B \succeq 0$, and $A \succ B $ if $A-B \succ 0$. The notations ``$\preceq$" and ``$\prec$" are similarly defined. A kernel $k \in L^1_{\mathrm{loc}}(\R_+; \C^{m \times m})$ is called \textit{completely monotone} with respect to the order induced by $\langle \cdot, \cdot \rangle$, if it is smooth on $(0,\infty)$ and satisfies
\[
 (-1)^n \left( \frac{d}{dt} \right)^n k(t) \succeq 0, \qquad \forall n \in \N.
\]

Note that in the one-dimensional case the above definition of complete monotonicity reduces to the classical one. For additional results on completely monotone functions on $\C^{m\times m}$ we refer to \cite[Chapter 5]{MR1050319}. We are now ready to provide a sufficient condition on how examples based on complete monotonicity can be constructed.
\begin{Proposition}
 Let $K = \widetilde{K}I_m$ with $\widetilde{K} \in L_{loc}^2(\R_+; \R)$ being completely monotone. Let $\beta \in \R^{m \times m}$ be symmetric with only strictly negative eigenvalues. Define $R_{\beta}, E_{\beta}$ by \eqref{eq: def Rbeta} and \eqref{eq: def Ebeta}, respectively. Then $R_{\beta},E_{\beta} \in L^1(\R_+;\R^{m\times m}) \cap L^2(\R_+; \R^{m \times m})$ are completely monotone.
\end{Proposition}
\begin{proof}
 It is clear that $K$ satisfies conditions (iv) -- (vi) from the definition of admissible parameters given in the introduction.
 Observe that $\langle v, (-\beta)v\rangle \geq 0$ for all $v \in \C^m$. Hence
 \[
  \langle v, (-1)^n K^{(n)}(t)(-\beta)v\rangle
  = (-1)^n\widetilde{K}^{(n)}(t) \langle v, (-\beta)v \rangle \geq 0
 \]
 shows that also $-K\beta$ is completely monotone.
 Thus, by \cite[Chapter 5, Theorem 3.1]{MR1050319}, $R_{\beta}$ is  completely monotone and integrable on $\R_+$. Hence $0 \preceq R_{\beta}(t) \preceq R_{\beta}(t_0)$ for $t\ge t_0 > 0$. Using $\beta^{\top} = \beta$ so that $R_{\beta}^{\top} = R_{\beta}$, we find that if $t\ge t_0 > 0$, then for $v \in \C^m$,
 \begin{align*}
     |R_{\beta}(t)v|^2  &=  \langle R_{\beta}(t)^{1/2}v, R_{\beta}(t)R_{\beta}(t)^{1/2}v \rangle
     \\ &\leq  \langle R_{\beta}(t)^{1/2}v, R_{\beta}(t_0)R_{\beta}(t)^{1/2}v \rangle
     \\ &\leq  |v|^2\| R_{\beta}(t_0)\|_{2}  \|R_{\beta}(t)^{1/2}\|_{2}^2
     \\ &\leq  |v|^2\| R_{\beta}(t_0)\|_{2}  \|R_{\beta}(t)\|_{2},
 \end{align*}
 where we have used $\| R_{\beta}(t)^{1/2}\|_2^2 = \sup_{|v|=1}\langle v, R_{\beta}(t)v\rangle \leq \| R_{\beta}\|_2$.
 This shows  $\|R_{\beta}(t)\|_2 \le \| R_{\beta}(t_0)\|_{2}$ for $t\ge t_0$. So
 \begin{align*}
     \int_0^{\infty}\|R_{\beta}(t)\|^2_2 dt  &\leq \int_0^{t_0}\| R_{\beta}(t)\|_{2}^2 dt +\| R_{\beta}(t_0)\|_{2} \int_{t_0}^{\infty} \|R_{\beta}(t)\|_{2} dt.
 \end{align*}
 Since $R_{\beta}$ is locally square integrable, the right-hand side is finite and hence $R_{\beta} \in L^2(\R_+; \R^{m \times m})$. Since $E_{\beta} = R_{\beta}(-\beta)^{-1}$, it follows that also $E_{\beta}$ is completely monotone and belongs to $L^1(\R_+;\R^{m\times m}) \cap L^2(\R_+; \R^{m \times m})$. This proves the assertion.
\end{proof}
Below we apply our results to the fractional kernel.
\begin{Example}
 Let $\beta \in \R^{m \times m}$ be symmetric with only strictly negative eigenvalues and $K(t) = \frac{t^{H-1/2}}{\Gamma(H + 1/2)}I_m$ with $H \in (0,1/2)$. Then $R_{\beta}, E_{\beta} \in L^1(\R_+; \R^{m \times m}) \cap L^2(\R_+; \R^{m \times m})$ are completely monotone and satisfy
 \begin{align*}
     \int_0^{\infty}R_{\beta}(t)dt = I_m, \qquad \int_{0}^{\infty}E_{\beta}(t)dt = (-\beta)^{-1}.
 \end{align*}
 \end{Example}
 \begin{proof}
  In view of previous Proposition it is clear that $R_{\beta}, E_{\beta}$ are completely monotone and satisfy the integrability conditions. The explicit formulas for the integrals can be obtained from the Laplace transforms. Namely,
  \begin{align*}
  \int_{0}^{\infty} R_{\beta}(t) dt
  &= \lim_{z \downarrow 0} \int_{0}^{\infty} e^{-z t}R_{\beta}(t) dt
  \\ &= \lim_{z \downarrow 0} \widehat{K_{\beta}}(z)(I_m + \widehat{K_{\beta}}(z))^{-1}
  \\ &= \lim_{z \downarrow 0} I_m-(I_m + \widehat{K_{\beta}}(z))^{-1}
  \\ &= \lim_{z \downarrow 0} I_m-(I_m  -z^{-H-1/2}\beta)^{-1}
  = I_m
  \end{align*}
  and hence $\int_0^{\infty}E_{\beta}(t)dt = \int_0^{\infty}R_{\beta}(t)(-\beta)^{-1}dt = (-\beta)^{-1}$.
 \end{proof}
The next example also covers the case where $\beta$ is not symmetric but  the fractional kernel is instead replaced by an integrable Gamma kernel.
\begin{Example}
 Let $K(t) = \frac{t^{H-1/2}}{\Gamma(H + 1/2)}e^{-\lambda t} I_m$ with $H \in (0,1/2)$, $\lambda > 0$, and let $\beta \in \R^{m \times m}$ be invertible. Then $R_{\beta}, E_{\beta}$ are integrable if and only if
 \begin{align}\label{eq: spectrum condition beta}
  \sigma(\beta) \cap \left \{ (z + \lambda)^{H + 1/2} \ : \ z \in \C, \ \ \mathrm{Re}(z) \geq 0  \right\} = \emptyset.
 \end{align}
  In such a case one has, by direct computations using Laplace transforms,
 \begin{align}\label{eq: 03}
  \int_{0}^{\infty} R_{\beta}(t) dt = (- \beta) (\lambda^{H+1/2} - \beta)^{-1}, \qquad
  \int_0^{\infty} E_{\beta}(t)dt = (\lambda^{H+1/2} - \beta)^{-1}.
 \end{align}
\end{Example}
\begin{proof}
 Since $-K\beta$ is integrable, the Paley-Wiener Theorem (see \cite[Chapter 2, Theorem 4.1]{MR1050319}) states that $R_{\beta}$ is integrable if and only if $\mathrm{det}(I_m - \widehat{K}(z)\beta) \neq 0$ holds for all $z \in \C$ with $\mathrm{Re}(z) \geq 0$.
 Computing the Laplace transform of $K$ gives $\widehat{K}(z) = (z + \lambda)^{-H - 1/2}I_m$. Hence $\mathrm{det}(I_m - \widehat{K}(z)\beta) \neq 0$  is equivalent to $\mathrm{det}((z+\lambda)^{H + 1/2}I_m - \beta) \neq 0$, i.e., $R_{\beta}$ is integrable if and only if \eqref{eq: spectrum condition beta} holds.
 Since $E_{\beta} = R_{\beta}(-\beta)^{-1}$, we see that  $E_{\beta} \in L^1(\R_+)$ is also equivalent to \eqref{eq: spectrum condition beta}. To prove \eqref{eq: 03}, we can use \eqref{Laplace eq for RB and KB} to get
 \[
 \int_{0}^{\infty} R_{\beta}(t) dt =\widehat{K_{\beta}}(0)(I_m + \widehat{K_{\beta}}(0))^{-1}  = (- \beta) (\lambda^{H+1/2} - \beta)^{-1}.
 \]
 The second inequality in \eqref{eq: 03} now follows again from $E_{\beta} = R_{\beta}(-\beta)^{-1}$.
\end{proof}
Note that condition \eqref{eq: spectrum condition beta} is satisfied, if
 $\sigma(\beta) \subset \{ z \in \C \ : \ \mathrm{Re}(z) < \lambda^{H+1/2} \}$.
Finally, we provide the asymptotics for the autocovariance function associated with the Gamma kernel in dimension $m = 1$, and hence prove the formula \eqref{asym autocovariance}) given in Theorem \ref{Theorem: VCIR 1dim}.
\begin{Example}
 Let $m=1$ and $(b, \beta,\sigma, K)$ be admissible parameters with $\beta < 0$, $\sigma > 0$ and $K(t) = \frac{t^{H-1/2}}{\Gamma(H+1/2)}e^{-\lambda t}$, where $H \in (0,1/2)$ and $\lambda \geq 0$. Then there exist  positive constants $h_0, c(h_0),C(h_0) $ such that for all $h \geq h_0$,
 \[
  c(h_0)h^{-(H+3/2)}e^{-\lambda h} \leq \int_0^{\infty}E_{\beta}(h+u)E_{\beta}(u)du \leq C(h_0)h^{-(H+3/2)}e^{-\lambda h}.
 \]
\end{Example}
\begin{proof}
 For $\alpha \in (0,1)$, define $e_{\alpha}(t) = t^{\alpha - 1}M_{\alpha}(-t^{\alpha})$, where $M_{\alpha}(z) = \sum_{n=0}^{\infty} z^n/\Gamma(\alpha n + \alpha)$ denotes the Mittag-Leffler function. It follows from \cite[Section A.1]{MR3905737} that
 $e_{\alpha}(t) \,\asymp\,  t^{-1-\alpha}$ as $t \to \infty$.
 By \cite[Table 1]{MR4019885}, we have
 $E_{\beta}(t) = |\beta|^{-1+\alpha^{-1}}e^{-\lambda t}e_{\alpha}(|\beta|^{1/\alpha} t)$
 with $\alpha = H + 1/2$. Hence we obtain
 $E_{\beta}(t) \,\asymp\, t^{-H - 3/2}e^{-\lambda t}$ as $t \to \infty$.
 Thus we find $h_0 \geq 1$ large enough and constants $c(h_0),C(h_0) > 0$ such that
 \begin{align*}
  \int_0^{\infty}E_{\beta}(h+u)E_{\beta}(u)du
  &\leq C(h_0)\int_0^{\infty} (h+u)^{-H-3/2}e^{-\lambda (h+u)} E_{\beta}(u)du
  \\ &\leq h^{-(H+3/2)}e^{-\lambda h} C(h_0)\int_0^{\infty}E_{\beta}(u)du
 \end{align*}
 holds for $h \geq h_0$. Similarly, for $h \geq h_0$, we obtain
 \begin{align*}
  \int_0^{\infty}E_{\beta}(h+u)E_{\beta}(u)du
  &\geq c(h_0)\int_0^{\infty} (h+u)^{-H-3/2}e^{-\lambda (h+u)} E_{\beta}(u)du
  \\ &\geq c(h_0)(2h)^{-(H+3/2)}e^{-\lambda h}\int_0^h e^{-\lambda u}E_{\beta}(u)du
 \\ &\geq c(h_0)2^{-(H+3/2)}h^{-(H+3/2)}e^{-\lambda h}\int_0^{h_0} e^{-\lambda u}E_{\beta}(u)du.
 \end{align*}
 Combining both estimates proves the assertion.
\end{proof}

\begin{Remark}
 Under same conditions as in the previous example, if $H = 1/2$, then $E_{\beta}(t) = e^{-(\lambda + |\beta|)t}$ and hence
 \[
  \int_0^{\infty} E_{\beta}(h+u)E_{\beta}(u)du = \frac{e^{-(\lambda + |\beta|)h}}{2(\lambda + |\beta|)}.
 \]
 Thus when $\lambda = 0$, we observe a phase transition from power-law to exponential decay in the asymptotics of the autocovariance function as $H \nearrow 1/2$.
\end{Remark}

Finally, below we provide a  method with which examples for non-scalar-valued kernels $K$ can be constructed.
\begin{Proposition} \label{prop: sufficient condition}
 Let $K = \mathrm{diag}(K_1, \dots, K_m) \in  L^1(\R_+; \R^{m\times m})$ and let $\beta \in \R^{m \times m}$.
 If $\|\beta\|_{2}\sum_{j=1}^m \| K_j \|_{L^1} < 1$, then $R_{\beta}$ is integrable.
\end{Proposition}
\begin{proof}
 Since $-K\beta$ is integrable, we can use again the Paley-Wiener Theorem
 to find that that $R_{\beta}$ is integrable if and only if
 $\mathrm{det}(I_m - \widehat{K}(z)\beta) \neq 0$,
 for all $z \in \C$ with $\mathrm{Re}(z) \geq 0$.
 Thus it suffices to show that $I_m - \widehat{K}(z)\beta$ is invertible for all $z$ with $\mathrm{Re}(z) \geq 0$.
 This is indeed the case, if $\| \widehat{K}(z)\beta\|_{2} < 1$.
 Estimating $\| \widehat{K}(z)\beta\|_{2} \leq \| \beta \|_{2} \sum_{j=1}^m \| K_j \|_{L^1}$ proves the assertion.
\end{proof}
The proof also provides a necessary condition for the integrability of $R_{\beta}$.
\begin{Remark}
Let $K = \mathrm{diag}(K_1, \dots, K_m) \in  L^1(\R_+; \R^{m\times m})$ and let $\beta \in \R^{m \times m}$. Let $l \in \N$ be the number of distinct eigenvalues $\rho_1,\dots, \rho_l$ of $\beta$. If $R_{\beta}$ is integrable, then
 \begin{align*}
  \rho_k(\widehat{K}_1(z), \dots, \widehat{K}_m(z))^{\top} \neq (1, \dots, 1)^{\top}
 \end{align*}
 holds for each $k = 1,\dots, m$ and $z \in \C$ with $\mathrm{Re}(z) \geq 0$.
\end{Remark}
\begin{proof}
 Arguing as before, $R_{\beta}$ is integrable if and only if $I_m - \widehat{K}(z)\beta$ is invertible. In particular, one has
 $(I_m - \widehat{K}(z)\beta)w \neq 0$ for $w \neq 0$.
 Let $k \in \{1,\dots, m\}$ and $w_k$ be an eigenvector for the eigenvalue $\rho_k$. Since $K$ is diagonal, $\widehat{K}$ is also diagonal, and hence $\widehat{K}(z)\beta w_k =
 \rho_k \widehat{K}(z)w_k \neq w_k$. This proves the assertion.
\end{proof}

\begin{Example}
 Let $(b,\beta, \sigma, K)$ be admissible parameters with $K_j(t) = \frac{t^{H_j - 1/2}}{\Gamma(H_j + 1/2)}e^{-\lambda_j t}$, where $H_j \in (0,1/2)$ and $\lambda_j > 0$. Then $\int_0^{\infty} K_j(t) dt = \lambda_j^{-H_j - 1/2}$.
 Therefore, for each $\beta$ we can find $\lambda_1,\dots, \lambda_m$ large enough such that Proposition \ref{prop: sufficient condition} is applicable.
\end{Example}

\section{Absolute continuity of the law}

In this section we provide sufficient conditions for
the distribution of the Volterra square-root process to be absolutely continuous with respect to the Lebesgue measure. Moreover, we establish also similar results  for the limiting distributions $\pi_{x_0}$.
For $f:\R^{m}\longrightarrow\R$ and $x,h\in\R^{m}$ set $\Delta_{h}f(x)=f(x+h)-f(x)$.
For $\lambda\in(0,1)$ the Besov space $B_{1,\infty}^{\lambda}(\R^{m})$
of order $\lambda$ and integrability $(1,\infty)$ consists of all
equivalence classes of functions $f$ with finite norm
\begin{align}\label{eq: Besov norm}
\|f\|_{B_{1,\infty}^{\lambda}}=\|f\|_{L^{1}}+\sup_{|h|\leq1}|h|^{-\lambda}\|\Delta_{h}f\|_{L^{1}}.
\end{align}
The following condition guarantees that the contribution from the noise term $\int_0^t K(t-s)\sigma(X_s)dB_s$ is non-degenerate on the event where $X$ belongs to the interior of the state space, see \eqref{lower bound}.
\begin{enumerate}
\item[(R)] There exists $\alpha \in [\gamma,2]$ and a constant $C_{*}>0$ such
that
\begin{equation} \label{condition R}
\int_{0}^{h} K_i(r)^2 dr \geq C_{*}h^{\alpha}, \qquad \forall h\in[0,1]
\end{equation}
holds for all $i = 1,\dots, m$.
\end{enumerate}
Note that for an admissible kernel $K$ satisfying \eqref{condition R} one necessarily has $\alpha\geq\gamma$.
\begin{Example}
 Let $K(t) = \frac{t^{H - 1/2}}{\Gamma(H + 1/2)}e^{-\lambda t}I_m$ with $H \in (0,1/2)$ and $\lambda \geq 0$. Then condition (R) holds for $\alpha = 2H = \gamma$. Thus all results below are applicable to this kernel.
\end{Example}
The following is our first main result for this section.
\begin{Theorem}\label{Theorem: regularity process}
 Let $X$ be the Volterra square-root process with admissible parameters $(b,\beta, \sigma, K)$.
 Suppose that condition (R) holds, and that
 \begin{align}\label{eq: minimum sigma}
  \sigma_* := \min\{ \sigma_1, \dots, \sigma_m \} > 0, \qquad \frac{\gamma}{\alpha} > \frac{2}{3}.
 \end{align}
 Denote by $\mu_{t}(dx)$ the distribution of $X_t$ at time $t\geq0$. Then for each $t > 0$ the finite measure
 $\mu_t^{*}(dx) = \min\{1, x_1^{1/2}, \dots, x_m^{1/2}\} \mu_t(dx)$
 has a density $\mu_{t}^{*}(x)$, and there exists $\lambda \in (0,1)$ such that for each $T > 0$,
 \[
 \|\mu_{t}^{*}\|_{B_{1,\infty}^{s}} \leq C(T)(1\wedge t)^{-\alpha/2}, \qquad t \in (0,T],\,s\in(0,\lambda],
 \]
where $C(T) > 0$ is a constant.
\end{Theorem}
\begin{proof}
 Here and below we denote by $C(T) > 0$ a constant which may vary from line to line.
 The proof relies essentially on an application of \cite[Lemma 2.1]{DF13}  to the finite measure $\mu_t^{*}$.

 \textit{Step 1.} For given $t>0$ and $\e\in(0,1\wedge t)$ define
\begin{align}
X_{t}^{\e} &= X_{t-\e} + \int_{t-\e}^{t}K(t-r)\left( b + \beta X_{t-\e} \right)dr \notag
 \\ & \ \ \  +\int_{0}^{t-\e}\left(K(t-r)-K(t-\e-r)\right)\left( b + \beta X_r\right)dr  \label{eq: approximation}
 \\ & \ \ \ + \int_{0}^{t-\e}\left(K(t-r)-K(t-\e-r)\right)\sigma(X_{r})dB_{r} \notag
 \\ &\ \ \ + \int_{t-\e}^{t}K(t-r)\sigma(X_{t-\e})dB_{r}. \notag
\end{align}
 By direct computation we find that
 \[
  X_t - X_t^{\e} = \int_{t - \e}^t K(t-r)\beta (X_r - X_{t-\e})dr + \int_{t-\e}^t K(t-r)\left( \sigma(X_r) - \sigma(X_{t-\e}) \right) dB_r.
 \]
 Using Proposition \ref{prop: Hoelder local} we obtain for the first term
 \begin{align*}
  &\ \E\left[ \left| \int_{t - \e}^t K(t-r)\beta (X_r - X_{t-\e})dr \right|^2 \right]
  \\ & \leq  \| \beta\|_{2}^2\left(\int_{t-\e}^{t}\|K(t-r)\|_{2}^{2}dr\right)\E\left[\int_{t-\e}^{t}|X_{r}-X_{t-\e}|^{2}dr\right]
  \\ &\leq \| \beta\|_{2}^2\left(\int_{t-\e}^{t}\|K(t-r)\|_{HS}^{2}dr\right)\E\left[\int_{t-\e}^{t}|X_{r}-X_{t-\e}|^{2}dr\right]
 \\ &\leq C(T)\e^{2\gamma + 1}.
 \end{align*}
 Similarly we obtain for the second term
 \begin{align*}
  &\ \E\left[ \left| \int_{t-\e}^t K(t-r)\left( \sigma(X_r) - \sigma(X_{t-\e}) \right) dB_r \right|^2 \right]
 \\ &\leq \E\left[\int_{t-\e}^{t}\|K(t-r)\|_{HS}^{2}\|\sigma(X_{r})-\sigma(X_{t-\e})\|_{HS}^{2}dr\right]
 \\ &\leq \sigma^* \sqrt{m} \int_{t-\e}^{t}\|K(t-r)\|_{HS}^{2}\E[|X_{r}-X_{t-\e}|]dr
 \\ &\leq C(T)\e^{3\gamma/2}.
 \end{align*}
 Combining both estimates shows that for each $T > 0$ there exists a constant $C(T) > 0$ such that
 \begin{align}\label{eq: approximation local}
  \E[|X_{t}-X_{t}^{\e}|^{2}]\leq C(T)\e^{3\gamma/2},\qquad t \in (0,T],\ \ \e\in(0,1\wedge t).
 \end{align}

 \textit{Step 2.} Let $\eta\in(0,1)$ and $0<t\le T$. Define $\rho(x) = \min\{1,x_1^{1/2},\dots, x_m^{1/2}\}$. In this step we prove that there exists a constant $C(T)>0$ such that for all $\phi\in C_{b}^{\eta}(\R^{m})$ and $h\in\R^{m}$ satisfying $|h|\leq1$ one has
 \begin{align}\label{eq: auxiliary estimate}
 \bigg|\E[ \rho(X_t) \Delta_{h}\phi(X_{t})]\bigg|
 \leq C(T)\|\phi\|_{C_{b}^{\eta}}\left(|h|^{\eta}\e^{3 \gamma /8}+\e^{3\eta \gamma/4}+|h|\e^{-\alpha/2}\right),
 \end{align}
 for $\e\in(0,1\wedge t]$. Observe that $\left|\E[ \rho(X_t)\Delta_{h}\phi(X_{t})]\right|\leq R_{1}+R_{2}+R_{3}$, where $R_{1},R_{2},R_{3}$ are given by
 \begin{align*}
 R_{1} &= \E\left[\left| \rho(X_t) - \rho(X_{t-\e})\right||\Delta_{h}\phi(X_{t})|\right],
 \\ R_{2} &= \E\left[\rho(X_{t-\e})\left|\Delta_{h}\phi(X_{t})-\Delta_{h}\phi(X_{t}^{\e})\right|\right],
 \\ R_{3} &= \left|\E[\rho(X_{t-\e})\Delta_{h}\phi(X_{t}^{\e})]\right|.
 \end{align*}
 To estimate the first term we use $|\rho(x) - \rho(y)| \leq m |x-y|^{1/2}$ to arrive at
 \begin{align*}
  R_{1} \leq m|h|^{\eta}\|\phi\|_{C_{b}^{\eta}}\E[|X_{t}-X_{t}^{\e}|^{1/2}]
  \leq C(T)|h|^{\eta}\|\phi\|_{C_{b}^{\eta}} \e^{3\gamma / 8}.
 \end{align*}
 For the second term we use $\rho(x) \leq 1$ to find that
 \begin{align*}
  R_{2} \leq 2\|\phi\|_{C_{b}^{\eta}}\E\left[ |X_{t}-X_{t}^{\e}|^{\eta} \right]
  \leq C(T)\|\phi\|_{C_{b}^{\eta}} \e^{3\eta \gamma/4}.
 \end{align*}
 Finally, for the last term we first note that $X_{t}^{\e}=U_{t}^{\e}+V_{t}^{\e}$,
 where $U_{t}^{\e}$ and $V_{t}^{\e}$ are given by
 \begin{align}\label{eq: approximation formula}
  \begin{cases} U_{t}^{\e} =& X_{t-\e}+\int_{t-\e}^{t}K(t-r)(b+\beta X_{t-\e})dr  \\
  & \qquad +\int_{0}^{t-\e}\left(K(t-r)-K(t-\e-r)\right)(b+\beta X_{r})dr
  \\ & \qquad   +\int_{0}^{t-\e}\left(K(t-r)-K(t-\e-r)\right)\sigma(X_{r})dB_{r},
  \\ V_{t}^{\e} =&\int_{t-\e}^{t}K(t-r)\sigma(X_{t-\e})dB_{r}.
  \end{cases}
 \end{align}
 Observe that $V_{t}^{\e}=\int_{0}^{\e}K(\e-r)\sigma(X_{t-\e})d\widetilde{B}_{r}$
 holds almost surely, where $\widetilde{B}_{r}=B_{t-\e+r}-B_{t-\e}$ is a
 new Brownian motion with respect to a new (shifted) filtration $\widetilde{\F}_{r}=\F_{t-\e+r}$.
 Hence $V_{t}^{\e}$ has, for fixed $t$ and $\e$, conditionally on
 $\F_{t-\e}$ a Gaussian distribution with mean zero and variance
 \begin{align*}
 \mathrm{var}(V_{t}^{\e}|\mathcal{F}_{t-\e}) &=\int_{0}^{\e}K(\e-r)\sigma(X_{t-\e})\sigma(X_{t-\e})^{\top}K(\e-r)^{\top}dr.
 \end{align*}
 Hence we obtain for each $x\in\R^{m}$
 \begin{align}
 \langle x,\mathrm{var}(V_{t}^{\e}|\F_{t-\e})x\rangle
 &= \int_{0}^{\e}\langle  K(\e-r)^{\top}x,\sigma(X_{t-\e})\sigma(X_{t-\e})^{\top}K(\e-r)^{\top}x\rangle dr \notag
 \\ &\geq \sigma_*^2 \rho(X_{t-\e})^2 \int_{0}^{\e}\langle x,K(r)K(r)^{\top}x\rangle dr \notag
 \\ &\geq \sigma_*^2 \rho(X_{t-\e})^2 C_* \e^{\alpha} |x|^2. \label{lower bound}
 \end{align}
 The product in front of $|x|^2$ in \eqref{lower bound} is strictly positive on $\{ \rho(X_{t-\e}) > 0 \}$.
 Hence the law of $V_{t}^{\e}$, when restricted to $\{ \rho(X_{t-\e}) > 0 \}$, has conditionally on $\F_{t-\e}$ a density $f_{t}^{\e}(z;X_{t-\e})$. Noting that $U_{t}^{\e}$ is $\F_{t-\e}$
 measurable, we find that
 \begin{align*}
 R_{3} &= \left|\E\left[\int_{\R^{m}} \rho(X_{t-\e})\1_{\{\rho(X_{t-\e}) > 0\}}\Delta_{h}\phi(U_{t}^{\e}+z)f_{t}^{\e}(z;X_{t-\e})dz\right]\right|
 \\ &= \left|\E\left[\int_{\R^{m}}\rho(X_{t-\e})\1_{\{\rho(X_{t-\e}) > 0\}}\phi(U_{t}^{\e}+z)\Delta_{-h}f_{t}^{\e}(z;X_{t-\e})dz\right]\right|
 \\ &\leq \|\phi\|_{C_{b}^{\eta}}|h|\E\left[ \rho(X_{t-\e})\1_{\{\rho(X_{t-\e}) > 0\}} \int_0^1\int_{\R^{m}}|\nabla f_{t}^{\e}(z-rh;X_{t-\e})|dzdr\right]
 \\ &\leq C(T)\|\phi\|_{C_{b}^{\eta}}\E\left[ \rho(X_{t-\e})\1_{\{\rho(X_{t-\e}) > 0\}}\frac{|h|}{\sqrt{\lambda_{\min}(\mathrm{var}(V_{t}^{\e}|\F_{t-\e}))}}\right]
 \\ &\leq C(T)\|\phi\|_{C_{b}^{\eta}}|h|\e^{-\alpha/2},
\end{align*}
where we have used that on $\{ \rho(X_{t-\e}) > 0 \}$,
\begin{align*}
\int_{\R^{m}}|\nabla f_{t}^{\e}(z-rh;X_{t-\e})|dz\leq\frac{C}{\sqrt{\lambda_{\min}(\mathrm{var}(V_{t}^{\e}|\F_{t-\e}))}}
\leq \frac{C(T)}{\rho(X_{t-\e})}\e^{-\alpha/2}
\end{align*}
for $\varepsilon\in(0,1\wedge t]$. Here $\lambda_{\min}(\cdot)$ denotes the smallest eigenvalue of a symmetric positive definite matrix. This proves \eqref{eq: auxiliary estimate}.

 \textit{Step 3.} Since $\frac{\gamma}{\alpha}> \frac{2}{3}$, we find a constant $a$ such that $\frac{4}{3\gamma}< a < \frac{2}{\alpha}$.
 For this choice of $a$, we find sufficiently small $\eta>0$ such that $1-\eta-a\alpha/2>0$.
 Letting $\e=|h|^{a}(1\wedge t)$ with $|h|\leq 1$,
 we find for any $\phi\in C_{b}^{\eta}(\R^{m})$
 \begin{align*}
 &\ \left|\E[ \rho(X_t)\Delta_{h}\phi(X_{t})]\right|
 \\ &\leq C(T)\|\phi\|_{C_{b}^{\eta}}\left( |h|^{\eta+3a\gamma / 8}+|h|^{3a \eta \gamma/4}+|h|^{1-a\alpha/2}\right)(1\wedge t)^{-\alpha/2}
 \\ &\leq C(T) \|\phi\|_{C_{b}^{\eta}}|h|^{\eta+\min\left\{ 3a\gamma / 8,\ 3a \eta \gamma/4 - \eta,1-\eta-a\alpha/2\right\} }(1\wedge t)^{-\alpha/2}
 \\ &\leq C(T)\|\phi\|_{C_{b}^{\eta}}|h|^{\eta+\lambda}(1\wedge t)^{-\alpha/2},
\end{align*}
where $\lambda= \min\{ 3a\gamma / 8, \eta(3a  \gamma/4 - 1),1-\eta-a\alpha/2\} > 0$. This shows that \cite[Lemma 2.1]{DF13} is applicable and hence proves the assertion.
\end{proof}

Note that $\mu_t^*$ is equivalent to $\mu_t$ on $\R_{++}^m = \{ x \in \R_+^m \ : \ x_i > 0, \ \forall i = 1,\dots, m \}$,  the interior of the state space. Hence we have the following corollary.

\begin{Corollary}
Assume that all assumptions of Theorem \ref{Theorem: regularity process} hold true. Then $\mu_t$ has a density when restricted to $\R_{++}^m$, i.e., $\mu_t |_{\R_{++}^m}$ has a density with respect to the Lesbegue measure.
\end{Corollary}
Note that, without further conditions only $\mu_t^*$ but not necessarily the density of $\mu_t|_{\R^m_{++}}$ belong to the Besov space $B_{1,\infty}^{\lambda}(\R^d)$ with some small $\lambda \in (0,1)$. To see this, let $X$ be the classical CIR process where $K(t) = 1$ and $m = 1$. Assume additionally $\beta = x_0 = 0$ and $\sigma, b > 0$. Then $\frac{2X_t}{\sigma^2 t}$ has $\chi^2$-distribution with $k := \frac{2b}{\sigma^2} > 0$ degrees of freedom (see also \cite[remark after Theorem 4]{MR3084047}), i.e. it has distribution $\nu(x)dx$ with
$\nu(x) = \1_{(0,\infty)}(x)\frac{x^{k/2 -1}e^{-x/2}}{2^{k/2}\Gamma(k/2)}$. In particular, there exists $\lambda \in (0,1)$ with $\nu \in B_{1,\infty}^{\lambda}(\R)$ iff $k = \frac{2b}{\sigma^2} > 2$, while $\min\{1,x^{1/2}\}\nu \in B_{1,\infty}^{\lambda}(\R)$ holds whenever $\lambda \in (0, k/2\wedge 1)$.
\begin{Corollary}\label{Theorem: regularity limit}
 Let $X$ be the Volterra square-root process with admissible parameters $(b,\beta, \sigma, K)$ and initial state $x_0$. Suppose that (R) and (K) are satisfied, and that
 $E_{\beta} \in L^1(\R_+; \R^{m \times m}) \cap L^2(\R_+; \R^{m \times m})$. Assume that
 \[
  \sigma_1, \dots, \sigma_m > 0 \ \text{ and } \
  \frac{\gamma}{\alpha} > \frac{2}{3}.
 \]
 Let $\pi_{x_0}$ be the limiting distribution of $X_t$ as $t\to \infty$. Then $\pi^{*}_{x_0}$ defined by
 $\pi^{*}_{x_0}(dy) = \min\{ 1, y_1^{1/2},\dots, y_m^{1/2}\} \pi_{x_0}(dy)$
 has a density $\pi_{x_0}^{*}(y)$ with respect to the Lebesgue measure, and there exists $\lambda \in (0,1)$ such that
 $\| \pi_{x_0}^{*}\|_{B_{1,\infty}^{s}} < \infty$ holds for $s \in (0,\lambda]$. In particular, $\pi_{x_0}$ is absolutely continuous with respect to the Lebesgue measure, when restricted to $\R_{++}^m$.
\end{Corollary}
\begin{proof}
 We follow the same steps of the previous proof, but now all constants can be chosen to be independent of $T$.
 Indeed, let $X_t^{\e}$ be defined as in \eqref{eq: approximation}. Using now Proposition \ref{prop: Hoelder global} we find,  with exactly the same estimates as in the proof of Step 1 in Theorem \ref{Theorem: regularity process}, $\E \left[ |X_{t}-X_{t}^{\e}|^2 \right]\leq C \e^{3\gamma /2}$, for $t>0$, and $\e\in(0,1\wedge t)$, where the constant $C$ is independent of $\e$ and $t$.

 Arguing similarly to Step 2 in the proof of Theorem \ref{Theorem: regularity process}, we obtain for all $\phi\in C_{b}^{\eta}(\R^{m})$ and $h\in\R^{m}$ satisfying $|h|\leq1$,
 \[
  \bigg|\E[ \rho(X_t) \Delta_{h}\phi(X_{t})]\bigg|
 \leq C\|\phi\|_{C_{b}^{\eta}}\left(|h|^{\eta}\e^{3 \gamma /8}+\e^{3\eta \gamma/4}+|h|\e^{-\alpha/2}\right).
 \]
 Hence arguing as in Step 3 we can find
  $a \in (\frac{4}{3}\frac{1}{\gamma}, \frac{2}{\alpha})$
 and $\eta \in (0,1)$ with $1 - \eta - a/\alpha > 0$ such that
 \begin{align*}
 \left|\E[ \rho(X_t)\Delta_{h}\phi(X_{t})]\right|
  \leq \|\phi\|_{C_{b}^{\eta}}|h|^{\eta+\lambda}(1\wedge t)^{-\alpha/2},
\end{align*}
 where $\lambda = \min\{ 3a\gamma / 8, \eta(3a  \gamma/4 - 1),1-\eta-a\alpha/2\} > 0$ and $\phi\in C_{b}^{\eta}(\R^{m})$.
 Hence we find that
\[
 \left| \int_{\R_+^m} \Delta_h\phi(x) \pi^*_{x_0}(dx) \right|
 = \lim_{t \to \infty} \left| \E\left[ \rho(X_t) \Delta_h \phi(X_t) \right] \right|
 \leq C\|\phi\|_{C_{b}^{\eta}}|h|^{\eta+\lambda}.
\]
This shows that \cite[Lemma 2.1]{DF13} is applicable and hence proves the assertion.
\end{proof}

\appendix

\section{Estimates on resolvent}

\begin{Lemma}\label{Ebeta Rbeta regularity}
 Let $K \in L_{loc}^2(\R_+; \C^{m \times m})$ and $B \in \C^{m \times m}$. Let $R_B$ be the resolvent of the second kind for $K_B(t) = - K(t)B$ and let $E_B$ be defined by \eqref{eq: def Ebeta}.
 Then for all $T>0$ and $s,t \in [0,T]$,
 \begin{align*}
    \int_s^t \| E_B(r)\|_2^2 dr &\leq 2\int_s^t \|K(r)\|_2^2 dr
  \\ & \ \ \ + 2 \| R_B \|_{L^1([0,T])}\int_0^s \left( \int_{s-u}^{t-u}\| K(r)\|_2^2 dr \right) \|R_B(u)\|_2 du
  \\ & \ \ \ + 2 \| R_B\|_{L^1([0,T])}^2 \int_0^{t-s} \| K(r)\|_2^2 dr.
 \end{align*}
 Moreover, for each $T>0$ and $h \in (0,1]$ we have
 \begin{align*}
   &\ \int_0^{T}\| E_{B}(t+h) - E_{B}(t)\|_{2}^2 dt
   \leq 3 \| R_B\|_{L^1([0,T+1])}^2 \int_0^{h}\|K(r)\|_2^2 dr
   \\ &\ \ \ + 3 \left(1 + \|R_B\|_{L^1([0,T])}\right)^2\int_0^T \|K(t+h)-K(t)\|_2^2dt.
 \end{align*}
\end{Lemma}
\begin{proof}
 Using \eqref{eq: def Ebeta} gives $\| E_{B} \|_{L^2([s,t])} \leq \| K\|_{L^2([s,t])} + \| R_{B} \ast K\|_{L^2([s,t])}$. To estimate the second term we first use Jensen's inequality and then the Fubini's Theorem to find that
 \begin{align*}
     \| R_{B} \ast K\|_{L^2([s,t])}^2
     &\leq \| R_B \|_{L^1([0,T])} \int_s^t \int_0^r \| R_B(u)\|_2 \| K(r-u)\|_2^2 dr du
     \\ &= \| R_B \|_{L^1([0,T])}\int_0^s \left( \int_{s-u}^{t-u}\| K(r)\|_2^2 dr \right) \|R_B(u)\|_2 du
     \\ &\ \ \ + \| R_B \|_{L^1([0,T])}\int_s^t \left( \int_0^{t-u}\| K(r)\|_2^2 dr \right) \| R_B(u)\|_2 du
     \\ &\leq \| R_B \|_{L^1([0,T])}\int_0^s \left( \int_{s-u}^{t-u}\| K(r)\|_2^2 dr \right) \|R_B(u)\|_2 du
     \\ & \ \ \ + \| R_B\|_{L^1([0,T])}^2 \int_0^{t-s} \| K(r)\|_2^2 dr.
 \end{align*}
 Combining these estimates proves the first inequality.
 For the second estimate we first note that
 $\| E_B(\cdot +h) - E_B(\cdot) \|_{L^2([0,T])}
     \leq \| I_1 \|_{L^2([0,T])} + \| I_2 \|_{L^2([0,T])} + \| I_3 \|_{L^2([0,T])}$ with $I_1 = K(t+h) - K(t)$, and
     \[
      I_2 = \int_0^t R_B(r)(K(t+h-r) - K(t-r))dr, \ \
      I_3 = \int_t^{t+h}R_B(r)K(t+h-r)dr.
     \]
Below we estimate the last two terms separately. For $I_2$ we obtain from Jensen's inequality
 \begin{align*}
   \|I_2 \|_{L^2([0,T])}^2 &\leq \| R_B\|_{L^1([0,T])} \int_0^T \int_0^t \| R_B(r)\|_2 \| K(t+h-r) - K(t-r)\|_2^2 dr dt
   \\ &\leq \| R_B\|_{L^1([0,T])}^2 \int_0^T \| K(t+h) - K(t)\|_2^2 dt.
 \end{align*}
 Similarly we obtain
 \begin{align*}
   \|I_3 \|_{L^2([0,T])}^2 &\leq \| R_B\|_{L^1([0,T+1])} \int_0^T \int_t^{t+h}\|R_B(r)\|_2 \|K(t+h-r)\|_2^2 dr dt
   \\ &= \| R_B\|_{L^1([0,T+1])} \int_0^T \int_0^h \| R_B(t+r)\|_2 \|K(h-r)\|_2^2 dr dt
   \\ &\leq \| R_B\|_{L^1([0,T+1])}^2 \int_0^h \|K(h-r)\|_2^2 dr
   \\ &= \| R_B\|_{L^1([0,T+1])}^2\int_0^{h}\|K(r)\|_2^2 dr.
 \end{align*}
 Combining all estimates proves the assertion.
\end{proof}
Below we specify the above estimates to the case of admissible parameters $(b,\beta, \sigma, K)$ with $B = \beta$. There we will also use the following observation: for any $t,s \geq 0$ with $s < t$ and $i = 1,\dots, m$, it holds the inequality
\[
 \int_s^t |K_i(r)|^2 dr\le \int_0^{t-s} |K_i(r)|^2 dr\le C_1 (t-s)^\gamma,
\]
since $K_i$ is admissible and thus non-increasing.
\begin{Lemma}\label{lemma: global Ebeta}
 Let $(b,\beta, \sigma, K)$ be admissible parameters and suppose that condition (K) holds.
 Suppose that $E_{\beta} \in L^1(\R_+; \R^{m \times m})$. Then there exists a constant $C > 0$ such that for all $t,s \geq 0$ with $0 \leq t-s \leq 1$, the following inequality holds:
 \begin{align*}
  &\ \int_s^t \| R_{\beta}(r)\|_2^2 dr
  + \int_s^{t}\|E_{\beta}(r)\|_2^2 dr
  \\ &\ \ \ + \int_0^{s}\| E_{\beta}(t-r) - E_{\beta}(s-r)\|_{HS}^2 dr
  + \int_s^{t}\|E_{\beta}(t-r)\|_{HS}^2 dr
  \leq C (t-s)^{\gamma}.
 \end{align*}
\end{Lemma}
\begin{proof}
 Lemma \ref{Ebeta Rbeta regularity} gives
 \begin{align*}
   &\  \int_s^t \| R_{\beta}(r)\|_2^2 dr
  + \int_s^{t}\|E_{\beta}(r)\|_2^2 dr
  \\ &\leq (1 + \|\beta\|_2^2) \int_s^t \|E_{\beta}(r)\|_2^2 dr
  \\ &\leq 2(1 + \|\beta\|_2^2) \int_s^t \|K(r)\|_2^2 dr
  \\ & \ \ \ + 2 (1 + \|\beta\|_2^2)\| R_{\beta} \|_{L^1(\R_+)}\int_0^s \left( \int_{s-u}^{t-u}\| K(r)\|_2^2 dr \right) \|R_{\beta}(u)\|_2 du
  \\ & \ \ \ + 2 (1 + \|\beta\|_2^2)\| R_{\beta}\|_{L^1(\R_+)}^2 \int_0^{t-s} \| K(r)\|_2^2 dr
  \\ &\leq 2(1 + \|\beta\|_2^2) m C_1\left( 1+ \|R_{\beta}\|_{L^1(\R_+)}^2\right) (t-s)^{\gamma}
 \end{align*}
 Similarly we obtain from Lemma \ref{Ebeta Rbeta regularity}
 \begin{align*}
     &\ \int_s^{t}\|E_{\beta}(t-r)\|_{HS}^2 dr
     \\ &\leq m \int_0^{t-s}\|E_{\beta}(r)\|_{2}^2 dr
     \\ &\leq 2m \int_0^{t-s} \|K(r)\|_2^2 dr
  + 2m \| R_{\beta}\|_{L^1(\R_+)}^2 \int_0^{t-s} \| K(r)\|_2^2 dr
     \\ &\leq 2m^2C_1(t-s)^{\gamma} + 2m^2 \| R_{\beta}\|_{L^1(\R_+)}^2 C_1 (t-s)^{\gamma}.
 \end{align*}
 Finally, the last term can be estimated by
 \begin{align*}
     &\ \int_0^{s}\| E_{\beta}(t-r) - E_{\beta}(s-r)\|_{HS}^2 dr
     \\ &\leq m \int_0^{s}\| E_{\beta}(t-s + r) - E_{\beta}(r)\|_{2}^2 dr
     \\ &\leq 3m (1+\|R_{\beta}\|_{L^1(\R_+)})^2 \int_0^s\|K(r+t-s)-K(r)\|_2^2dr
     \\ &\ \ \ + 3m \|R_{\beta}\|_{L^1(\R_+)}^2\int_{0}^{t-s}\|K(r)\|_2^2 dr
     \\ &\leq 3m^2 (1+\|R_{\beta}\|_{L^1(\R_+)})^2 C_3(t-s)^{\gamma}
     + 3m^2 \|R_{\beta}\|_{L^1(\R_+)}^2 C_1 (t-s)^{\gamma}.
 \end{align*}
 Combining all estimates proves the assertion.
\end{proof}

\section*{Acknowledgements}
The authors would like to thank the referees for a careful reading of this manuscript which lead to a great improvement of this work.

\bibliographystyle{amsplain}
\phantomsection\addcontentsline{toc}{section}{\refname}\bibliography{Bibliography}

\end{document}